\documentclass[12pt]{amsart}    
\usepackage[colorlinks]{hyperref}
\usepackage{amsmath,amssymb,latexsym, amscd}
\usepackage{mathabx}
\usepackage{exscale, cite, eps fig, graphics}
\usepackage{mathtools}
\usepackage{enumitem}
\usepackage{lipsum}
\usepackage{graphicx}
\usepackage{subfig}
\usepackage{soul}

\usepackage{array}
\usepackage{multirow}

\newcommand\MyBox[2]{
  \fbox{\lower0.75cm
    \vbox to 1.7cm{\vfil
      \hbox to 1.7cm{\hfil\parbox{1.4cm}{#1\\#2}\hfil}
      \vfil}%
  }%
}

\calclayout

\renewcommand{\geq}{\geqslant}
\renewcommand{\leq}{\leqslant}

\newcommand{\g}{\gamma}

\renewcommand{\b}{\beta}

\newcommand{\R}{\mathbb{R}}

\newcommand{\T}{\mathbb{T}}
\newcommand{\C}{\mathbb{C}}

\newcommand{\Dz}{\partial_z}
\newcommand{\calM}{\mathcal{M}}

\DeclareMathOperator{\csch}{csch}
\newtheorem{theorem}{Theorem}[section]
\newtheorem{lemma}[theorem]{Lemma}
\newtheorem{proposition}[theorem]{Proposition}
\newtheorem{remark}[theorem]{Remark}
\newtheorem{definition}[theorem]{Definition} 
\newtheorem{corollary}[theorem]{Corollary}

\newtheorem*{main-theorem}{Main Theorem}
\newtheorem*{remark*}{Remark}

\newtheorem{hypothesis}[theorem]{Hypothesis}
\numberwithin{equation}{section}

\title{Modulational Instability in the Ostrovsky Equation and Related Models}
\author[Bhavna]{Bhavna$^\ast$}
\author[Johnson]{Mathew A. Johnson$^\dagger$}
\author[Pandey]{Ashish~Kumar~Pandey$^\ast$}
\email{bhavnai@iiitd.ac.in, matjohn@ku.edu, ashish.pandey@iiitd.ac.in}
\address{$^\ast$Department of Mathematics, IIIT Delhi,  110020, India\\
$^\dagger$ Department of Mathematics, University of Kansas, 1460 Jayhawk Boulevard, Lawrence, KS 66049, USA}
\date{\today}

\begin{document}

\maketitle
\begin{abstract}

We study the modulational instability of small-amplitude periodic traveling wave solutions in a dispersion generalized Ostrovsky equation. Specifically, we investigate the invertibility of the 
associated linearized operator in the vicinity of the origin and derive a modulational instability index that depends on the dispersion and nonlinearity.  
For the classical Ostrovsky equation, we recover the well-known Lighthill condition for modulational instability of small-amplitude periodic traveling waves,
and further provide a rigorous connection of the Lighthill condition to the spectral instability of the underlying wave.  
Our results and methodologies further apply to a wide-class of Ostrovsky type
models that incorporate various dispersive effects.  As such, we present new results illuminating the effects of rotation on various full-dispersion models
arising in the study of weakly nonlinear surface water waves.
\end{abstract}

\section{Introduction}

In 1978,  Lev Ostrovsky \cite{Ostrovsky1978NonlinearOcean} proposed the following equation, which is now commonly referred to by his name,
\begin{equation}\label{E:Ost}
(u_t + \beta u_{xxx}+(u^2)_x)_x = \g u ,
\end{equation}
as a model for the unidirectional propagation of weakly nonlinear surface and internal waves of small amplitude in a rotating fluid in the long-wavelength regime.
 Here, $\beta$ is the dispersion coefficient, which determines the type of dispersion, and $\gamma$ is the coriolis parameter which measures the rotational effect of the earth\cite{Galkin1991OnFluid}. The Ostrovsky equation \eqref{E:Ost} also occurs as a model for the
propagation of short pulses in nonlinear media and as a model for magneto-acoustic solitary waves in a plasma\cite{Gilman1995ApproximateEquation}. Wave dynamics in relaxing medium \cite{Vakhnenko1999High-frequencyMedium} are described by this equation and its reduced form (the reduced Ostrovsky equation) taking $\beta=0$. The Gardner-Ostrovsky equation, which contains both quadratic and cubic nonlinearities, can be obtained by further generalising the Ostrovsky equation \cite{Holloway1999AZone}.  A brief account of the history and applications
of the Ostrovsky equation can be found in \cite{Stepanyants2019NonlinearWaves}.

There has also been interest in the so-called 
generalized Ostrovsky equation, given by
\begin{equation}\label{eq:gost2}
\left( u_t+\beta u_{xxx}+f(u)_x\right)_x-\gamma u=0,~x \in \R
\end{equation}
which generalizes the nonlinearity of the dispersive medium while 
satisfying a power-like scaling condition in the long wavelength regime: see, for example, \cite{Holloway1999AZone}. 
As with most nonlinear dispersive equations, 
the existence and dynamics of periodic traveling or solitary waves are of special interest in both \eqref{E:Ost} and \eqref{eq:gost2}. The existence of solitary waves and their stability in the Ostrovsky equation and its generalizations have been demonstrated in \cite{Lu2012OrbitalEquation, Liu2004StabilityEquation, Levandosky2007StabilityEquation,PS20}. In \cite{Hakkaev2017PeriodicStability, Geyer2017SpectralEquation,JP16}, periodic traveling waves of
the Ostrovsky equation with general nonlinearity \eqref{eq:gost2},
as well as the generalized reduced Ostrovsky equation  (corresponding to $\beta=0$ above), have been formed for small values of $\g$ and have been shown to be spectrally stable to periodic perturbations of the same period as the wave. 
In \cite{Grimshaw2008Long-timeEquation}, 
 it was shown through weakly nonlinear analysis that
quasi-monochromatic wave trains \cite{Stepanyants2019NonlinearWaves,Whitfield2014Rotation-inducedWaves} of the Ostrovsky equation are modulationally unstable in a relatively “short-wave” range $k_ch<kh\ll1$ and absent in the longer wave range $kh<k_ch$, where $k$ is the wave-number, $h$ is the basin depth and $k_c$ is the critical wave number. 
This weakly nonlinear analysis has been extended to more general models with cubic nonlinearity (the so-called Gardner–Ostrovsky equation \cite{Whitfield2015Wave-packetEquation} and the Shrira equation\cite{Nikitenkova2015ModulationalDispersions}). 
Finally, the authors' in \cite{Bhavna2022High-FrequencyEquation} have recently investigated that the periodic traveling waves of the Ostrovsky equation \eqref{E:Ost} suffers high-frequency instability if $k>\sqrt[4]{4\gamma/\beta}$ for positive values of $\beta$ and no high-frequency instability for when\footnote{This should be 
compared to Corollary \ref{c:1} below, where modulational instability is detected for $\beta<0$.} $\beta<0$.

In the absence of rotational effects, the models \eqref{E:Ost} reduces to the well-studied Korteweg-de Vries (KdV) equation
\[
u_t-\beta u_{xxx}+\left(u^2\right)_x=0,
\]
which is a canonical model for the unidirectional propagation of weakly nonlinear, small amplitude water waves in the long wavelength regime: see  \cite{Lannes2013TheAsymptotics}
and references therein.  It is well known-however, that while the KdV equation explains well long-wave phenomena in a channel of water --  traveling solitary and periodic waves, for example --
it fails to exhibit many high-frequency phenomena such as wave breaking -- the evolutionary formation of bounded solutions with infinite gradients -- and peaking -- the existence 
of bounded, steady solutions with a singular point, such as a peak or cusp.  This of course should not be surprising, as the phase velocity associated with the linear 
part of the KdV equation poorly approximates that of the water wave equations outside the long wavelength regime.  Indeed,  the (non-dimensional) phase speed for unidirectional
water waves can be shown to expand as
\[
c_{ww}(k):=\sqrt{\frac{\tanh(k)}{k}} =1-\frac{1}{6}k^2+\mathcal{O}(k^4),~~|k|\ll 1
\]
where here $k$ is the frequency of the wave.  Thus, in the long wavelength regime $|k|\ll 1$ the water wave phase speed $c_{ww}(k)$ agrees up to $\mathcal{O}(k^2)$
with the phase speed associated to the KdV equation.

In an effort to find a simple mathematical equation that could explain water wave phenomena outside of the long-wavelength regime, Whitham proposed the model
\begin{equation}\label{e:whitham}
u_t+\beta\mathcal{M}_{ww}u_x+\left(u^2\right)_x=0
\end{equation}
where here $\mathcal{M}_{ww}$ is a Fourier multiplier with symbol
\[
\widehat{\mathcal{M}_{ww}f}(k) = c_{ww}(k)\hat{f}(k).
\]
The model \eqref{e:whitham}, now referred to as the Whitham equation, balances both the full phase speed for unidirectional water waves with a canonical shallow water nonlinearity
and hence, Whitham conjectured,  one might expect it to be capable of predicting both breaking and peaking of water waves.  This has recently been seen to be the case.
Indeed, \eqref{e:whitham} has recently been seen to exhibit both wave breaking \cite{Hur_breaking} and peaking \cite{EK_global,EW_WhithamConj}.  Additionally, \eqref{e:whitham}
was shown in \cite{Hur2015ModulationalWaves} to bear out the famous Benjamin-Feir, or modulational, instability of small amplitude periodic traveling waves: see also the related
numerical work in \cite{Carter24,SKC_WhithamNumerics} on the stability of large amplitude periodic waves.  Taken together, it seems clear that, regardless of its rigorous connection to the full
water wave problem\footnote{The relevance of the Whitham equation as a model for water waves has recently been studied in \cite{MKD,BKN}, where it was found to perform better than the KdV 
and BBM equations in describing surface water waves in the intermediate and short wave regime.}, the dispersion generalized model \eqref{e:whitham} admits many interesting high-frequency features known to exist in the full water wave problem.

Given the success  of \eqref{e:whitham} in describing water wave phenomena outside of the long-wavelength regime, a number of recent
works have aimed to study such models that incorporate additional physical effects such as surface tension, constant vorticity, as well
as models allowing for bidirecitonal wave propagation: see, for example, \cite{HJ_SurfaceTension,CJ_num1,EJC,EJMR,JW,JTW,Carter_2018} and references therein.
Interestingly, these works that focus on stability may often be carried out by considering a \emph{dispersion generalized Whitham equation}
\begin{equation}\label{e:whitham2}
u_t+\beta\mathcal{M}u_x+\left(u^2\right)_x=0,
\end{equation}
where here the Fourier multiplier $\mathcal{M}$
\begin{equation}\label{e:M}
\widehat{\mathcal{M}f}(k)=m(k)\hat{f}(k)
\end{equation}
need only have a symbol $m(k)$ that satisfies a few simple non-degeneracy, smoothness and growth assumptions.  Results for specific models, or incorporating specific physical
properties, can then be immediately ascertained by substituting for $\mathcal{M}$ a specific Fourier multiplier whose symbol agrees with the associated phase velocity.
For example,  with the choice $m(k) = 1-|k|^\alpha (\alpha > 1)$  \eqref{e:whitham2} is the Fractional KdV (fKdV) equation, $m(k) = 1-|k|$ is the Benjamin-Ono (BO) equation, 
$m(k)=k\coth k$ is the Intermediate Long wave (ILW) equation and $m(k)=\sqrt{\tanh k/k}$ is the Whitham equation \eqref{e:whitham}.

Continuing in this spirit, this current work seeks to study the existence and stability of periodic traveling wave solutions in the following
dispersion-generalized  Ostrovsky (gOst) equation 
\begin{equation}\label{eq:gost}
\left( u_t+\beta\calM u_x+(u^2)_x\right)_x-\gamma u=0,~ \gamma>0, ~\beta \in \R\setminus\{0\},
\end{equation}
where here $\mathcal{M}$ is a Fourier multiplier operator as in \eqref{e:M}
which satisfies the following assumptions:
\begin{hypothesis}\label{h:m} The multiplier $m(k)$ in \eqref{e:M} is assumed to satisfy the following:
\begin{enumerate}
    \item[(H1)] $m$ is real valued, even and without loss of generality, $m(0)=1$;
    \item[(H2)] there exists constants $C_1,C_2>0$ and $\alpha\geq -1$ such that  
    \[
    C_1 k^\alpha \leq m(k) \leq C_2 k^\alpha
    \]
for $ k \gg 1$. 
\end{enumerate}
Additionally, throughout we will assume the frequencies $k>0$ considered satisfy the following:
\begin{enumerate}
\item[(H3)] for each fixed $n=2,3,\ldots$ we have 
\[
k^2 \left( m(kn)-m(k) \right) \neq\dfrac{\g (n^2-1)}{\b n^2}
\]
for all $k>0$.
\end{enumerate}
\end{hypothesis}

\begin{remark}
Hypotheses (H1)-(H3) above will be essential for the proof of the existence of small-amplitude periodic traveling waves of \eqref{eq:gost}: see
Section \ref{s:2} below.  In particular, we note that hypotheses (H3) rules out the resonance between the fundamental mode and a higher harmonic. 
for more details, see the discussion directly preceding the statement of Theorem \ref{T:sol}.  

Also, we note in the case $m(k)=1-k^2$, corresponding to the classical Ostrovsky equation, the non-resonance condition (H3) holds for all $k>0$
if $\beta<0$ while if $\beta>0$ it requires that 
\[
k\notin\left\{\left(\frac{\gamma}{\beta n^2}\right)^{1/4}:n\in\mathbb{N},~~n\geq 2\right\},
\]
thus excluding a countable number of frequencies tending to zero as $n\to\infty$.
\end{remark}

The model \eqref{eq:gost} can thus be thought of as an extension of \eqref{e:whitham2} that incorporates rotational effects. Notice \eqref{eq:gost} recovers the standard
Ostrovsky equation \eqref{E:Ost} by choosing $\mathcal{M}=-\partial_x^2$.  Further, one can take $m(k)=1+|k|^\alpha$, $m(k) = 1+|k|$,
$m(k)=k\coth(k)$ and $m(k) = \sqrt{\tanh(k)/k}$ to produce variants of the fKdV, BO, ILW, and Whitham equations that incorporates rotating background fluids.
Further, note that since $m(k)$ is even, if $u(x,t)$ satisfies \eqref{E:Ost} for a particular choice of $\beta$ and  $\gamma$, then $-u(x,-t)$ satisfies \eqref{E:Ost} 
with $\beta$ and  $\gamma$ replaced by $-\beta$ and $-\gamma$ respectively. Due to this symmetry, we restrict ourselves to the case of $\gamma > 0$.

\

In this article, we investigate the modulational instability of small-amplitude periodic traveling waves of the gOst equation \eqref{eq:gost}. Perturbing the stokes waves, that is, small-
amplitude, one-dimensional periodic traveling waves of permanent form and constant velocity, by functions bounded in space and exponential in time usually yields a spectral problem whose spectral elements characterize the time growth rates of the perturbation. This collection of spectral elements is referred as the stability spectrum of theses waves. But the scenario is not same here. In order to study the modulational instability, we need to study the spectrum for small values of Floquet exponent. The differential operator becomes singular for very small values of Floquet exponent and hence the results found by replacing the study of invertibility by the study of the spectrum will not be uniform for small values of Floquet exponent, that is, the set of values of amplitude for which these results are valid depends upon Floquet exponent, which should not be the case. Hence, we can not study modulational instability by converting the invertibility problem into the spectral problem and we
instead study the invertibility problem directly. 

To state our main result,  note that the dispersion relation associated to \eqref{eq:gost} is given by
\begin{align}\label{e:dispersion}
    \omega(k)=\b km(k)+\dfrac{\g}{k}
\end{align}
while the corresponding phase and group velocities are given by 
\begin{equation}\label{e:gvel}
    c_p(k) := \frac{\omega(k)}{k} = \b m(k)+\dfrac{\g}{k^2}\quad \text{and} \quad c_g(k):=\dfrac{d\omega}{dk}=\b (m(k)+k m^\prime (k))-\dfrac{\g}{k^2},
\end{equation}
respectively.   We now state our main result, providing a criterion governing the modulational instability of periodic traveling waves of the 
generalized Ostrovsky equation \eqref{eq:gost} which is as follows.

\begin{theorem}[Modulational instability index]\label{t:1} A $2\pi/k$-periodic traveling wave
of \eqref{eq:gost} with sufficiently small amplitude is modulationally unstable if
\begin{equation}\label{e:cd}
   \Delta(k):=(c_p(k)-c_p(2k)) \frac{dc_g(k)}{dk}<0
\end{equation}
where $c_p$ and $c_g$ are phase and group velocities respectively given in \eqref{e:gvel}. It is modulationally stable if $\Delta(k)>0$.  
\end{theorem}

\begin{remark}
Our precise definition of modulational stability and instability is provided in Definition \ref{D:MI} in Section \ref{s:3} below.
\end{remark}

\begin{remark}\label{R:Lighthill}
We note that the instability condition \eqref{e:cd} is precisely the same as the well-known Lighthill criteria for modulational instability of 
small-amplitude periodic traveling wave solutions in the context of the Ostrovsky equation.  Indeed, in \cite{Grimshaw2008Long-timeEquation,GSA16}  the authors
use formal asymptotic methods (so-called modulation theory) to show that if $u$ is a small amplitude, weakly nonlinear periodic 
solution of the Ostrovsky equation with an asymptotic expansion of the form
\[
u(x,t) = \left(A e^{i\theta}+c.c.\right) + \left(A_2 e^{2i\theta}+c.c.\right)+\ldots
\]
where c.c. denotes the complex conjugate of the preceding term, $\theta=kx-\widetilde{\omega}(|A|,k) t$, where $\widetilde{\omega}(|A|,k)$ is the nonlinaer dispersion relation, 
and $|A|=\varepsilon\ll 1$ is slowly varying (in space and time) 
with $A_2=\mathcal{O}(\varepsilon^2)$, then the leading order term $A$ satisfies the effective nonlinear Schr\"odinger (NLS) equation
\begin{equation}\label{e:NLS}
iA_t+\frac{1}{2}\omega''(k) A_{XX}-\omega_2(k)|A|^2A=0,
\end{equation}
where here $X=x-c_g(k)t$, $\omega(k)$ is the dispersion relation \eqref{e:dispersion} for the Ostrovsky equation and where $\omega_2$
is identified as the $\mathcal{O}(\varepsilon^2)$ correction to the dispersion relation 
\[
\widetilde{\omega}(|A|,k) = \omega(k)+\varepsilon^2\omega_2+\mathcal{O}(\varepsilon^4)
\]
in the amplitude of the wave.  The Lighthill criteria then says that the small weakly nonlinear periodic traveling wave solution is modulationally unstable
provided that the NLS equation \eqref{e:NLS} is focusing, i.e. provided that
\begin{equation}\label{e:Lighthill}
\omega''(k)\omega_2(k)<0.
\end{equation}
In Remark \ref{R:Lighthill2} below, we will show that the Lighthill condition \eqref{e:Lighthill} agrees precisely with the rigorous modulational instability index
\eqref{e:cd} stated in Theorem \ref{t:1}.  In this way, in the case of the Ostrovsky equation, our work can be seen as a rigorous justification of the (formal) Lighthill
condition \eqref{e:Lighthill} at the level of the (rigorous) spectral stability of the underlying small-amplitude periodic traveling wave.  Of course,
our analysis also applies to the wider class of dispersion generalized Ostrovsky equations given in \eqref{eq:gost2}.
\end{remark}

From Theorem~\ref{t:1}, $\Delta(k)$ can change sign using two mechanishms - first when $c_p(k)=c_p(2k)$, that is, phase velocities of first and second harmonic coincide, and second when $\frac{dc_g(k)}{d k}=0$, that is, group velocity has a critical value. As an immediate corollary of Theorem~\ref{t:1}, we obtain modulational instability in the Ostrovsky equation \eqref{E:Ost}.

\begin{corollary}[Modulational instability]\label{c:1}
For a fixed $\g>0$, a $2\pi/k$-periodic traveling wave of the Ostrovsky equation \eqref{E:Ost} with sufficiently small amplitude is modulationally unstable if $k>k_c$ where
   \begin{equation*}
\begin{cases}
k_c = \left(\dfrac{\g}{3|\b|}\right)^{1/4} & \quad \text{ if } \b >0, \\
k_c = \left(\dfrac{\g}{4|\b|}\right)^{1/4} & \quad \text{ if } \b <0
\end{cases}
\end{equation*}
and it is modulationally stable otherwise.
\end{corollary}

\begin{figure}[t]
\begin{center}
(a) \includegraphics[scale=0.4]{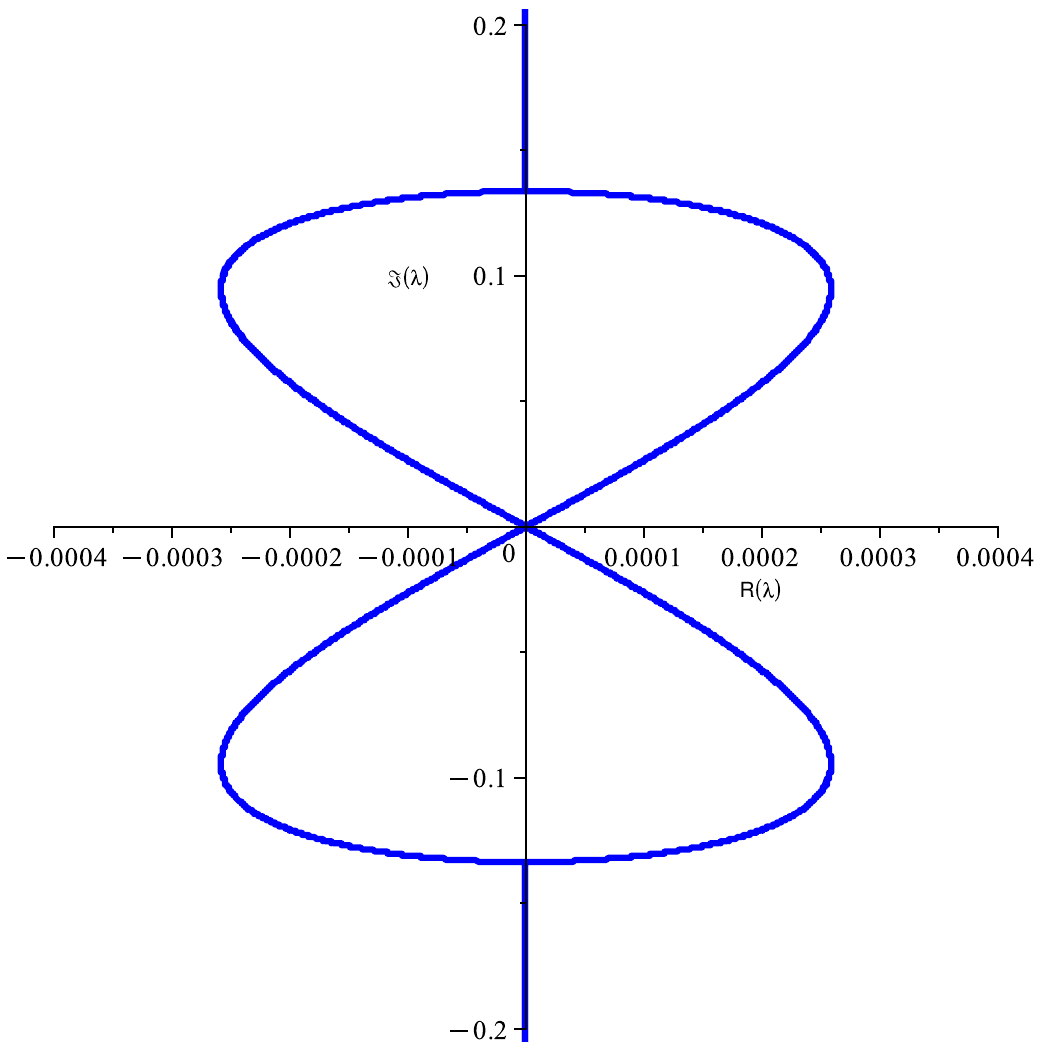}\qquad (b) \includegraphics[scale=0.4]{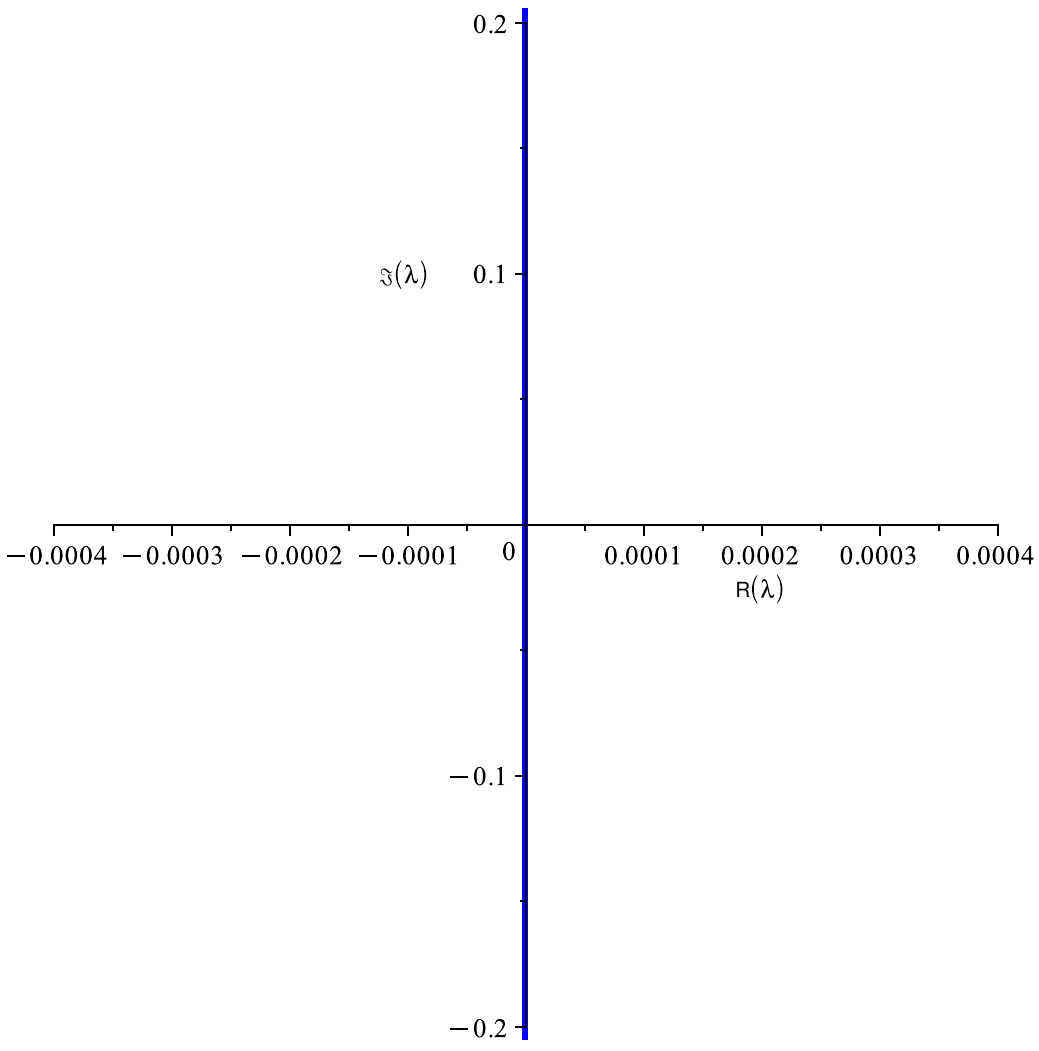}
\caption{A numerical determination of the $\lambda\in\mathbb{C}$ with $|\lambda|\ll 1$ such that the linear operator $\mathcal{T}_{k,a}^\lambda$ fails
to be invertible with bounded inverse on $L^2(\mathbb{R})$ for the classical Ostrovsky equation (corresponding to $m(k)=1-k^2$).  
In both figures, we took $\beta=1$ and $\gamma=1$, for which Corollary \ref{c:1} gives $k_c=0.76$.  Taking an amplitude $a=0.05$, in 
(a) we took $k=0.8>k_c$ Theorem \ref{t:1}.  Similarly, again taking $a=0.05$, in (b) we took $k=0.72<k_c$ and so the spectral stability in a 
neighbhrood of the origin is again consistent with Theorem \ref{t:1}.
}\label{F:Ost_Spec}
\end{center}
\end{figure}

For $\b=1$ and $\g=1$, the value of $k_c$ in Theorem~\ref{t:1} is approximately $0.76$ which agrees with critical wavenumber obtained by Whitfield and Johnson in \cite{Whitfield2014Rotation-inducedWaves}.  See also Figure \ref{F:Ost_Spec}  where modulational instability and stability is demonstrated numerically 
for cases $k>k_c$ and $k<k_c$, respectively.
In addition, Whitfield and Johnson mention that the instability is caused as group velocity has a critical value at $k_c$ which in accordance with our analysis for $\beta>0$. 
In our analysis, we will also see  when $\beta<0$ that modulational instability is caused by coinciding phase velocities of first and second harmonic at $k_c$.

\begin{remark}\label{r:singular_limit1}
From Corollary \ref{c:1} we see that, regardless of the sign of $\beta$, the critical frequency $k_c$ satisfies
\[
k_c\sim \gamma^{1/4},~~{\rm for}~~0<\gamma\ll 1.
\]
Thus, one may think this suggests that all periodic traveling wave solutions of the KdV equation (corresponding essentially to $\gamma=0$) exhibit a modulational instability.
This is, in fact not the case, as it is actually known that all periodic traveling waves in the KdV equation are modulationally stable (in fact, spectrally stable to general localized
and bounded perturbations) \cite{HK08,Bronski2010TheEquation,BJK_11,Bottman2009KdVStable}.  This emphasizes the singular nature of the $\gamma\to0^+$ limit, and results
about the associated $\gamma=0$ model can not be directly inferred from taking $\gamma\to 0^+$ in our analysis.
\end{remark}

\begin{remark}\label{r:reduced_limit1}
Also from Corollary \ref{c:1} we see that the critical frequency $k_c$ satisfies
\[
k_c\sim|\beta|^{-1/4},~~{\rm as}~~\beta\to 0.
\]
This suggests that periodic traveling wave solutions in the \emph{reduced} Ostrovsky equation, corresponding to \eqref{E:Ost} with $\beta=0$, are necessarily
modulationally stable.  This prediction is consistent with the work \cite{JP16}, where the authors establish that $T$-periodic traveling wave solutions
of the reduced Ostrovsky equation are nonlinearlly orbitally stable to $NT$-periodic perturbations for every $N\in\mathbb{N}$, i.e. they are orbitally stable
to subharmonic perturbations.  
\end{remark}

This article is organized as follows.
In Section~\ref{s:2}, the existence of a family of periodic traveling waves is shown using a standard argument based on implicit function theorem and Lyapunov-Schmidt reduction. Moreover, we obtain a small-amplitude expansion of these periodic traveling waves. We linearize \eqref{eq:gost} about the obtained periodic traveling wave and examine the invertibility of the linearized operator resulting in a modulational instability index in Section~\ref{s:3}. Using modulational instability index calculated in Section~\ref{s:3}, we continue by studying 
modulational instability for various
relevant choices of the dispersive operator $\mathcal{M}$ in \eqref{eq:gost}.  Specifically, we apply our result to the classical (and fractional) Ostrovsky equations, as well 
as rotational versions of various physically relevant surface water wave models, including the Whitham equation for water waves (both with and without surface
tension) as well as the Intermediate Long-Wave equation.  We conclude our study by additionally adding capillary effects to the classical Ostrovsky equation and to the
Ostrovsky variant of the Whitham equation, in an effort to understand the effects of both fluid rotation and surface tension.  Note that in the 
non-rotational case, the effects of surface tension was studied in \cite{HJ_SurfaceTension}.

\subsection*{Notations}\label{sec:notations}
The following notations are going to be used throughout the article. Let us assume $L^2(\mathbb{R})$ denotes the set of real or complex-valued, 
Lebesgue measurable functions $f(x)$ over $\mathbb{R}$ such that
\[
\|f\|_{L^2(\mathbb{R})}=\Big(\frac{1}{2\pi}\int_\R |f|^2~dx\Big)^{1/2}<+\infty \quad 
\]
and $L^2(\mathbb{T})$ denotes the space of $2\pi$-periodic, measurable, real or complex-valued functions over $\mathbb{R}$ such that
\[
\|f\|_{L^2(\mathbb{T})}=\Big(\frac{1}{2\pi}\int^{2\pi}_0 |f|^2~dx\Big)^{1/2}<+\infty. 
\]
For $f \in L^1(\mathbb{R})$, the Fourier transform of $f$ is written as $\hat{f}$ and defined by 
\[
\hat{f}(t)=\frac{1}{\sqrt{2\pi}}\int_{\R} f(x)e^{-itx}dx
\]
It follows from Parseval Theorem that if $f\in L^2(\mathbb{R})$ then $\|\hat{f}\|_{L^2(\mathbb{R})} = \|{f}\|_{L^2(\mathbb{R})}$.
Moreover, for any $s\in \mathbb{R}$, let $H^s(\mathbb{R})$ consist of tempered distributions such that 
\[
\|f\|_{H^s(\mathbb{R})} = \left(\int_{\R}(1+|t|^2)^s|\hat{f}(t)|^2dt\right)^{\frac{1}{2}} < +\infty
\]
Furthermore, the $L^2(\mathbb{T})$-inner product is defined as
\begin{equation}\label{def:i-product}
\langle f,g\rangle=\frac{1}{2\pi}\int^{2\pi}_{0} f(z)\overline{g}(z)~dz
=\sum_{n\in\mathbb{Z}} \widehat{f}_n\overline{\widehat{g}_n},
\end{equation}
where here, for a given $h\in L^2(\mathbb{T})$, 
\[
\widehat{h}_n=\frac{1}{2\pi}\int_0^{2\pi}h(z)e^{-inx}dx
\]
denotes the $n$th Fourier coefficient of $h$.  For any $s>0$, the (periodic) Sobolev space $H^s(\mathbb{T})$ is the subspace of $L^2(\mathbb{T})$ defined via the norm
\[
\left\|f\right\|_{H^s(\mathbb{T})}^2 = \widehat{f}_0^2 + \sum_{n\in\mathbb{Z}}|n|^{2s}|\widehat{f}_n|^2,
\]
and we let $H^s_{\rm even}(\mathbb{T})$ denote the subspace of $H^s(\mathbb{T})$ consisting of even functions.
Let $H^\infty(\mathbb{T})=\bigcap_{k=1}^\infty H^k(\mathbb{T})$. Moreover, the commutator of two operators acting on a Hilbert space is defined by the following
\[
[F , G] := F \circ G - G \circ F.
\]

\

\noindent
{\bf Acknowledgments:}  The work of MAJ was partially funded by the NSF under grant DMS-2108749. Bhavna and AKP are supported by the Science and Engineering Research Board (SERB), Department of Science and Technology (DST), Government of India under grant 
SRG/2019/000741. Bhavna is also supported by Junior Research Fellowships (JRF) by University Grant Commission (UGC), Government of India.
The authors would also like to thank the anonymous referees for their helpful insights, references and suggestions.

\section{ Asymptotically small-amplitude periodic traveling waves}\label{s:2}

To begin, we seek \emph{periodic traveling wave} solutions of \eqref{eq:gost}.  Here and throughout our work, we will always assume the symbol $m(\cdot)$
associated with the Fourier multiplier $\mathcal{M}$ satisfies (H1)-(H3) in Hypothesis \ref{h:m}.  To this end  we make a change of variables
 $z:=k(x-ct)$, where $k>0$ is the wave number and $c$ is the speed of the wave, and note that if $u$ is a $2\pi/k$-periodic traveling wave solution of $\eqref{eq:gost}$ then 
 the function $w(z):=u(k(x-ct),t)$  is a $2\pi$-periodic solution of 
\begin{align}\label{E:w}
    -ck^2 w''+\b k^2\calM_k w''+k^2(w^2)''-\g w=0,
\end{align}
where here $\mathcal{M}_k$ is a Fourier multiplier satisfying
\[
\mathcal{M}_ke^{inz} = m(kn)e^{inz}~~{\rm for}~~n\in\mathbb{Z}.
\]
Note that $m(k)$ is assumed to be real-valued and even. Consequently, \eqref{E:w} is invariant under translation ($z\mapsto z+z_0$) and $z\mapsto -z$ and therefore, we may assume that $w$ is even. Also, since \eqref{E:w} does not possess scaling invariance, we may not a priori assume that $k=1$. 
In fact, the modulational instability results obtained below depends on the wavenumber $k$ of the background periodic wave.

For fixed $\b$ and $\g$, define the operator\footnote{Here, $\alpha$ is as in (H2).} $F:H_{\rm even}^{\alpha+2}(\mathbb{T})\times \R \times \R^+ \to L^2(\mathbb{T})$ as
\begin{align*}
    F(w,c;k)=-ck^2 w''+\b k^2\calM_k w''+k^2(w^2)''-\g w.
\end{align*}
Note that $F$ is well defined by a elementary Sobolev embedding argument.
We seek a non-trivial $2\pi$-periodic solution $w$ of \eqref{E:w} in $H_{\rm even}^{\alpha+2}(\mathbb{T})$ with $c \in \R$ and $k>0$ such that 
\begin{equation*}\label{E:F1}
    F(w,c;k)=0.
\end{equation*}
Note that if $w\in H^{\alpha+2}(\mathbb{T})$ solves \eqref{E:F1} then by a Sobolev inequality we have $\mathcal{M}_kw''\in H^\alpha(\mathbb{T})$.
Therefore, by (H2) in Hypothesis \ref{h:m}, $w\in H^{2(\alpha+1)}(\mathbb{T})$.  By a bootstrap argument it follows that solutions $w\in H^{\alpha+2}(\mathbb{T})$
of \eqref{E:F1} necessarily satisfy $w\in H^\infty(\mathbb{T})$.  

Now, to study \eqref{E:F1} note that $F(0,c;k)=0$ for all $c\in\mathbb{R}$ and $k>0$ and
\[
\partial_wF(0,c;k) = -ck^2\partial_z^2+\beta k^2\mathcal{M}_k\partial_z^2-\gamma
\]
so that, in particular, for each $c\in\mathbb{R}$, $k>0$ and $n\in\mathbb{N}$ we have
\[
\partial_wF(0,c;k)\cos(nz) = \left(c(kn)^2-\beta (kn)^2m(kn)-\gamma\right)\cos(nz).
\]
It follows that 
\[
\ker\left(\partial_wF(0,c_0;k)\right) = {\rm span}\left(\cos(z)\right)
\]
provided that
\[
c=c_0=\frac{\gamma}{k^2}+\beta m(k).
\]
and that $k>0$ is chosen so that that\footnote{While the choice of $c_0$ guarantees that $\cos(z)$ is in the kernel of $\partial_w F(0,c;k)$, the restriction on $k$ guarantees
$\partial_w F(0,c_0;k)\cos(nz)\neq 0$ for any $n\in\mathbb{N}$ with $n\geq 2$.  That is, the restriction on $k$ guarantees that the kernel of $\partial_wF(0,c_0;k)$ is simple.}
\begin{equation}\label{e:k_cond}
k^2 \left( m(kn)-m(k) \right) \neq\dfrac{\g (n^2-1)}{\b n^2}~~~~{\rm for~any}~~n\in\mathbb{N},~~n\geq 2,
\end{equation}
i.e. $k$ should be chosen so that (H3) holds.
Using a Lyapunov-Schmidt argument, one can thus establish the existence of a one-parameter family of non-trivial, even solutions 
$(w(a;k)(\cdot),c(a;k))$ of $F(w,c;k)$ bifurcating
from $w\equiv 0$ and $c=c_0$ and defined for $|a|\ll 1$ provided that $k>0$ satisfies the non-resonance condition \eqref{e:k_cond}.  
This existence argument is elementary and follows the same lines as those in \cite{EK_local,Johnson2013StabilityEquations}, 
and is hence omitted here.  A key feature of the solutions $(w(a;k)(\cdot),c(a;k))$, however, is that they depend analytically on the parameter $a$ for $|a|\ll 1$.  Exploiting that fact,
the next result further establishes analytic expansions for these solutions valid for all $|a|\ll 1$.

\begin{theorem}\label{T:sol}
Suppose that the symbol $m(\cdot)$ in \eqref{e:M} satisfies hypotheses (H1)-(H2).
For all wavenumbers $k>0$ satisfying hypothesis (H3) there exists a one parameter family of solutions of \eqref{E:w}
 given by $u(x,t)=w(a;k)(k(x-c(a;k)t))$ for $a \in \R$ and $|a|$ sufficiently small; $w(a;k)(\cdot)$ is $2\pi$-periodic, 
 even and smooth in its argument, and $c(a;k)$ is even in $a$; $w(a;k)(\cdot)$ and $c(a;k)$ depend analytically on $a$ and $k$. Moreover, 
\begin{align}\label{E:w_ansatz}
    w(a;k)(z)=a\cos(z) + a^2A_2\cos 2z + a^3A_3\cos 3z+ O(a^4),
\end{align}
and
\begin{align}\label{e:c}
    c(a;k)=c_0+a^2c_2+O(a^4)
\end{align}
as $a \to 0$, where
\[c_0=\frac{\g}{k^2}+\b m(k)\]
and
\[
c_2=A_2, \quad A_2 = \dfrac{2k^2}{3\g+4\b k^2(m(k)-m(2k))} \quad and \quad  A_3 = \dfrac{9k^2A_2}{8\g+9\b k^2(m(k)-m(3k))}.
\]
\end{theorem}

\begin{proof}
As this argument is standard, here we only sketch the details.  The existence of the solutions $(w(a;k)(\cdot),c(a;k))$ for a fixed $k>0$ 
follows by an elementary Lyapunov-Schmidt argument: for similar arguments,
see, for example, \cite{EK_local}.  As a result, for a fixed $k>0$ it follows that the solutions $(w(a,k)(\cdot),c(a,k))$ are analytic in $a$ for $|a|\ll 1$ and hence may be expanded 
as\footnote{The relation $c(a;k)=c(-a;k)$ follows from the details of the Lyapunov-Schmidt argument discussed above.  See, for example, \cite{EK_local,Johnson2013StabilityEquations}.}
\begin{equation}\label{e:expand1} 
\left\{\begin{aligned}
w(a;k)(z)&=a\cos(z)+a^2w_2(k)(z)+a^3w_3(k)(z)+\mathcal{O}(a^4)\\
c(a;k)&=c_0+c_2(k)a^2+c_4(k) a^4+ \mathcal{O}(a^6)
\end{aligned}\right.
\end{equation}
Substituting these expansions into the profile equation \eqref{E:w} yields a hierarchy of compatability conditions indexed by the order of the small parameter $a$.
By construction, the first non-trivial equation occurs at $\mathcal{O}(a^2)$, which reads

\begin{equation}\label{e:solve1}
\partial_wF(0,c;k)w_2= 2k^2\cos{2z}
\end{equation}
Requiring, via the Fredholm Alternative, that  the right-hand side of \eqref{e:solve1} be orthogonal to the kernel of the (symmetric) operator $\partial_wF(0,c;k)$ yields 
the stated equation for $c_2=c_2(k)$, and then subsequently solving \eqref{e:solve1} results in the stated formula for the function $w_2$.  Continuing
to higher order equations in $a$, using again the Fredholm alternative and subsequently solving the resulting equation, yields the remainder of the formulae.  We omit
 the remainder of the details, and instead refer the interested reader to \cite{Johnson2013StabilityEquations}[Lemma 2.1], for instance.
\end{proof}

\begin{remark}
We note, in particular, that the solutions $	w(a;k)(z)$ constructed above all satisfy a mean-zero constraint, i.e. they satisfy $\int_0^{2\pi}w(z)dz=0$.  This of course
is a property of all localized or periodic traveling wave solutions of the generalized Ostrovsky equation considered here, as integrating
\eqref{eq:gost} over a period (for periodic traveling waves) or on the whole line (for localized traveling waves) yields the mean-zero requirement since $\gamma\neq 0$.
\end{remark}

\begin{remark}\label{R:Lighthill2}
From Theorem \ref{T:sol}, we see that the nonlinear dispersion relation $\omega(a;k)=kc(a;k)$ can be expanded as
\[
\omega(a;k) = \omega(k) + a^2\frac{2k^3}{3\gamma+4\beta k^2\left(m(k)-m(2k)\right)}+\mathcal{O}(a^4).
\]
Noting that
\[
c_p(k)-c_p(2k) = \frac{3\gamma}{4k^2}+\beta\left(m(k)-m(2k)\right)
\]
and that $\frac{dc_g}{dk}(k)=\omega''(k)$ by definition,
it follows that the Lighthill condition given in \eqref{e:Lighthill} can be expressed as
\[
\omega''(k)\omega_2(k) = \frac{k}{2\left(c_p(k)-c_p(2k)\right)}\cdot\frac{dc_g}{dk}(k),
\]
and hence the sign of the rigorous modulational instability index from Theorem \ref{t:1} agrees exactly with the Lighthill condition \eqref{e:Lighthill}.
\end{remark}

\section{Modulational Instability Index}\label{s:3}

 Throughout this section, let $w=w(x;a,k)$ with $k>0$ and $|a|\ll 1$ be a small amplitude $2\pi/k$-periodic traveling wave solution of \eqref{eq:gost} with
wave speed $c=c(a;k)$, whose existence follows from Theorem \ref{T:sol} above.  The goal of this section is to study
the modulational stability of the wave $w$.

Linearizing \eqref{eq:gost} about $w$ in the spatial  frame of reference $z=k(x-ct)$ we arrive at the linear evolution equation
\begin{align*}
    k(v_t - ck v_z +\b k\calM_k v_z+2k(w v)_z)_z=\g v.
\end{align*}
governing the perturbation $v(z,t)$.
We seek a solution of the form $v(z,t) = e^{\frac{\lambda}{k} t} \Tilde{v}(z)$, $\lambda\in \mathbb{C}$ and $\Tilde{v} \in L^2(\R)$ to arrive at the equation
\begin{align}\label{E:opt}
  \mathcal T^\lambda_{k,a} \Tilde{v} := (\lambda\partial_z +k^2\partial^2_z(-c + \b \calM_k + 2w )- \g)\Tilde{v} = 0,
\end{align}
where here $\mathcal{T}^\lambda_{k,a}:L^2(\mathbb{R})\to L^2(\mathbb{R})$ is considered as a closed, densely defined linear operator.

\begin{definition} \label{def:stable}
The periodic traveling wave solution $w$ of \eqref{eq:gost} is spectrally stable  with respect to square integrable perturbations if the operator $\mathcal T^\lambda_{k,a}$ is invertible
on $L^2(\mathbb{R})$ for every $\lambda \in \mathbb{C}$ with $\Re(\lambda)>0$.  Otherwise, it is deemed to be spectrally unstable.
\end{definition}

\begin{remark}
Note that since \eqref{E:opt} is invariant under the transformation $(v,\lambda)\mapsto (\bar{v},\bar{\lambda})$, as well as the transformation $(z,\lambda)\mapsto (-z,-\lambda)$,
the set of $\lambda\in\mathbb{C}$ where the operator $\mathcal{T}^\lambda_{k,a}$ fails to be invertible is symmetric with respect to reflections about the real and imaginary axes.
Consequently, $w$ is spectrally stable if and only if $\mathcal{T}^\lambda_{k,a}$ is invertible for all $\lambda\in\mathbb{C}$ with $\Re(\lambda)\neq 0$.
\end{remark}

Since the coefficients of the operator $\mathcal T^\lambda_{k,a}$ are $2\pi$-periodic, we can use the Floquet theory such that all solutions of \eqref{E:opt} in $L^2(\R)$ or $C_b(\R)$ are of the form $\Tilde{v}(z)=e^{i\xi z}V(z)$ where $\xi\in\left(-1/2,1/2\right]$ is the Floquet exponent and $V$ is a $2\pi$-periodic function.  As a  result, we get the following.

\begin{lemma}\label{lem:equiv}
The linear operator $\mathcal T^\lambda_{k,a}$ is invertible with bounded inverse on $L^2(\R)$ if and only if the linear operators
\begin{equation}
\mathcal T^\lambda_{k,a, \xi} := \lambda(\partial_z+i\xi) +k^2(\partial_z+i\xi)^2(-c + \b e^{i\xi z}\calM_ke^{-i\xi z} + 2w )- \g \end{equation}
acting in $L^2(\mathbb{T})$ with domain $H^{\alpha+2}(\mathbb{T})$ are invertible for all $\xi\in\left(-\frac12,\frac12\right]$. Moreover, $\mathcal T^\lambda_{k,a, \xi}$ is invertible in $L^2(\mathbb{T})$ if and only if zero is not an $L^2(\mathbb{T})$-eigenvalue of $\mathcal T^\lambda_{k,a, \xi}$.
\end{lemma}
\begin{proof}
We refer readers to \cite{Haragus2008STABILITYEQUATION,Haragus2011TransverseEquation,Johnson2013StabilityEquations} for detailed proofs in similar situations.  
It is straightforward to adapt those results to our case.
\end{proof}

Lemma~\ref{lem:equiv} reduces the invertibility problem \eqref{E:opt} in $L^2(\mathbb{R})$ to a one-parameter family of invertibility problems
\begin{align}\label{E:opt1}
  \mathcal T^\lambda_{k,a,\xi} \phi := (\lambda(\partial_z+i\xi) +k^2(\partial_z+i\xi)^2(-c + \b e^{i\xi z}\calM_ke^{-i\xi z} + 2w )- \g)\phi = 0
\end{align} 
in $L^2(\mathbb{T})$ indexed by $\xi\in \left(-\frac12,\frac12\right]$. 
From Definition \ref{def:stable}, it follows that the periodic traveling wave $w$ is spectrally unstable with respect to square integrable perturbations in $L^2(\mathbb{R})$ 
if and only if for some $\xi\in\left(-\frac12,\frac12\right]$ the operator $\mathcal T^\lambda_{k,a,\xi}$ acting on $L^2(\mathbb{T})$ has non-trivial kernel for some $\lambda\in\C$  
with $\Re(\lambda)>0$. 
While determining the full set of $\lambda$ for which the operator $\mathcal{T}^\lambda_{k,a,\xi}$ is invertible for all $\xi\in\left(-\frac12,\frac12\right]$ is already a daunting task,
in this work our focus is on a particular subclass of the possible instabilities.

\begin{definition}\label{D:MI}
A periodic traveling wave solution $w(a;k)$ of \eqref{eq:gost} is said to be modulationally stable if the associated linear operators $\mathcal{T}^\lambda_{a,k,\xi}$
is invertible on $L^2(\mathbb{T})$ for all $|(\lambda,\xi)|\ll 1$ with $\Re(\lambda)\neq 0$.  Otherwise, the solution $w(a;k)$ is modulationally unstable.
\end{definition}

\begin{remark}
In essence, the above states that the periodic traveling wave $w(a;k)$ is modulationally stable if the operator $\mathcal{T}^\lambda_{a,k}$ is invertible
for all $\lambda$ off the imaginary axis in a sufficiently small neighborhood of $\lambda=0$.  It follows that if the solution $w$ is modulationally
unstable, then it is indeed spectrally unstable in the sense of Definition \ref{def:stable}.  However, being modulationally stable does not automatically
imply spectral stability since the latter requires that $\mathcal{T}^\lambda_{a,k}$ is invertible for \emph{all} $\lambda\in\mathbb{C}$ with $\Re(\lambda)\neq 0$, not just
invertible for such $\lambda$ near the origin.  
\end{remark}

\begin{remark}
We note that modulational instabilities is a fundamental feature of many nonlinear systems, including those arising in the modeling of nonlinear optics as well
as surface water waves.  The connection between such ``spectral" modulational instabilities, as described above, and the dynamic instability of periodic traveling wave
solutions to slow modulations (via Whitham's theory of modulations) has been studied in many works, including \cite{JP20,Bronski2016ModulationalType,JZ10}.
In addition to the references mentioned above, the reader can consult \cite{W74} for a mathematical and physical discussion.
\end{remark}

We now list some properties of the operator $\mathcal T^\lambda_{k,a,\xi}$, which
may be readily verified by direct calculation.

\begin{proposition}[Symmetric Property]\label{p:1}
The operator $\mathcal T^\lambda_{k,a,\xi}$ acting on $L^2(\T)$ has following properties:
\begin{enumerate}
    \item $\mathcal T^\lambda_{k,a,\xi}(z)=\overline{\mathcal T^{-\overline{\lambda}}_{k,a,\xi}(-z)}$,
    \item $\mathcal T^\lambda_{k,a,\xi}=\overline{\mathcal T^{\overline{\lambda}}_{k,a,-\xi}}$.
    \end{enumerate}
\end{proposition}

 We now set forth the study of the $L^2(\T)$-kernel of the operator $\mathcal T^{\lambda}_{k,a,\xi}$ for $\xi\in\left(-\frac{1}{2},\frac{1}{2}\right]$ and $|a|$ sufficiently small. 
Note that, thanks to Proposition~\ref{p:1}(2), it is sufficient to take $\xi\in[0,1/2]$.
Since $k$ is fixed, in what follows, we denote $\mathcal T^{\lambda}_{k,a,\xi}$ by $\mathcal T^{\lambda}_{a,\xi}$.

We begin by discussing the case $a=0$, corresponding to the trivial solution $w=0$.  A straightforward Fourier calculation yields that
\begin{equation}
    \mathcal T^\lambda_{0,\xi}e^{inz}= \left(i\lambda(n+\xi)+\g((n+\xi)^2-1) +\b k^2(n+\xi)^2( m(k) - m(k(n+\xi)) )\right)e^{inz}=0 
\end{equation}
for all $n\in \mathbb{Z}$ and $\xi\in \left[0,1/2\right]$.  
The kernel of $\mathcal T^\lambda_{0,\xi}$ is thus non-trivial when 
\begin{equation}\label{e:Omega}
    \lambda=i\left(\g\left(n+\xi-\dfrac{1}{n+\xi}\right)+n\b k^2( m(k) - m(k(n+\xi)) )\right)=:i\Omega_{n,\xi}, \quad n\in \mathbb{Z}.
\end{equation}
and hence the trivial solution $w=0$  
of \eqref{eq:gost} is spectrally stable to square integrable perturbations as expected.  
Moreover, for $|a|$ small, because of Proposition~\ref{p:1}, values of $\lambda$ in \eqref{e:Omega} will bifurcate to leave imaginary axis only when two of them collide on imaginary axis. Therefore, we obtain instability for small $|a|$ if for some $\lambda$ in \eqref{e:Omega}, the kernel of $\mathcal T^{\lambda}_{0,\xi}$ is at least two-dimensional. Since we are only concerned with modulational (in)stability, we will only consider small values of $|\xi|$ and $|\lambda|$.

In particular, notice that the two values $\Omega_{1,\xi}$ and $\Omega_{-1,\xi}$ collide at $\lambda=0$ when $\xi=0$: for all other $n\in\mathbb{N}$
and $\xi$ we have $\Omega_{n,\xi}\neq 0$.  
Furthermore, the two-dimensional generalized kernel for $\mathcal{T}^{0}_{0,0}$ can be continued into a two-dimensional critical subspace
\[
\Sigma_{0,\xi}=\ker\left(\mathcal{T}^{i \Omega_{1,\xi}}_{0,\xi}\right)\oplus \ker\left(\mathcal{T}^{i \Omega_{-1,\xi}}_{0,\xi}\right)
\]
with ($\xi$-independent) orthogonal basis
\begin{equation}\label{e:basis1}
\phi_1(z) = \cos(z),\quad \phi_2(z)=\sin(z).
\end{equation}
For all other values of $n\neq -1,1$, 
 the kernel of $\mathcal T^{i\Omega_{n,0}}_{0,0}$ is one-dimensional and therefore, from Proposition~\ref{p:1}(1), can not lead to instability
for $|(a,\xi)|\ll 1$.  We thus aim to track how the values
$\Omega_{\pm 1,0}$ bifurcate from the origin for $|(a,\xi)|\ll 1$.  
To this end, we note that, for $|a|$ and $\xi$ small, the operator $\mathcal T^{\lambda}_{a,\xi}$ is a perturbation of $\mathcal T^{\lambda}_{0,0}$ with
\[
\|\mathcal T^{\lambda}_{a,\xi}-\mathcal T^{\lambda}_{0,0}\|_{L^2(\T)\to L^2(\T)}=O(|a|+|\xi|) 
\]
uniformly in operator norm as $|a|,\xi\to 0$. Consequently, for $|(a,\xi)|\ll 1$ there can be only two values $\lambda=i\Omega_{\pm 1,a,\xi}$ in a sufficiently
small neighborhood of the origin where the operator $\mathcal{T}^\lambda_{a,\xi}$ fails to be invertible and, further, the functions
\[
(a,\xi)\mapsto\Omega_{\pm 1,a,\xi}
\]
are analytic in $(a,\xi)$ for $|(a,\xi)|\ll 1$ and limit back to $\Omega_{\pm 1,0}=0$ as $(a,\xi)\to (0,0)$.  Further, one has a two-dimensional critical subspace
\[
\Sigma_{a,\xi}=\ker\left(\mathcal{T}^{i \Omega_{1,a,\xi}}_{a,\xi}\right)\oplus \ker\left(\mathcal{T}^{i \Omega_{-1,a,\xi}}_{a,\xi}\right)
\]
which is an analytical continuation of that found at $a=0$ above.  

Our goal is now to track the critical values $\lambda(a,\xi)=i\Omega_{\pm 1,a,\xi}$ for $|(a,\xi)|\ll 1$.
To this end, our strategy is essentially to project the operator equation $\mathcal{T}^{\lambda}_{a,\xi}v=0$ onto the two-dimensional critical subspace
$\Sigma_{a,\xi}$ above.  More precisely, we will compute a suitable basis $\{\phi_j(z;a,\xi)\}_{j=1,2}$ for $\Sigma_{a,\xi}$ and then compute the $2\times 2$
matrix 
\begin{align}\label{eq:bmat1}
\mathcal{B}_{a,\xi}^\lambda = \left[B_{ij}\right]_{i,j=1}^2~~{\rm with}~~ B_{ij} =  \frac{\langle \mathcal T^\lambda_{a,\xi}\phi_i(z;a,\xi),\phi_j(z;a,\xi)\rangle}{\langle\phi_i(z;a,\xi),\phi_i(z;a,\xi)\rangle}. 
\end{align}
The critical values $\lambda(a,\xi)$ are then found by solving the algebraic equation
\begin{align}\label{eq:cheq1p}
 \det(\mathcal{B}^\lambda_{a,\xi} ) = 0
\end{align}
for the variable $\lambda$.  
It remains then to find a suitable basis for the critical subspace $\Sigma_{a,\xi}$, and then to compute the appropriate projections above.

To compute a basis for $\Sigma_{a,\xi}$ that is compatible  with $\Sigma_{0,\xi}$, note that differentiating the profile equation \eqref{E:w} with respect to 
$z$ and $a$ gives
\[
\mathcal T^{0}_{a,0} (\partial_z w) = 0, \quad \mathcal T^{0}_{a,0} (\partial_a w) = k^2c_a\partial_z^2 w.
\]
Using the expansions in Theorem \ref{T:sol}, we thus obtain a normalized basis for the critical subspace $\Sigma_{a,0}$, i.e.
the  generalized kernel of $\mathcal{T}^0_{a,0}$, as
\begin{equation}\label{e:eig1}
    \phi_1(z;a,0)=-\frac{1}{a}(\Dz w)(z)=\sin z+2aA_2\sin 2z+3a^2A_3\sin 3z+O(a^4)
\end{equation}
and 
\begin{equation}\label{e:eig2}
    \phi_2(z;a,0)=(\partial_a w)(z)=\cos z+2aA_2\cos 2z+3a^2A_3\cos 3z+O(a^4),
\end{equation}
where here the values $A_j$ are as in Theorem \ref{T:sol}.  These functions provide asymptotic extensions for the generalized kernel of $\mathcal{T}^0_{0,0}$ and, in fact, they provide
an asymptotic extension for the ($\xi$-independent) basis of the critical subspace $\Sigma_{0,\xi}$ provided in \eqref{e:basis1}.  By spectral perturbation theory, it follows
that the functions $\phi_1(\cdot;a,0)$ and $\phi_2(\cdot;a,0)$ continue into a $\xi$-dependent basis for the critical subspace $\Sigma_{a,\xi}$ for $|(a,\xi)|\ll 1$.  We note, however,
that as in \cite{Johnson2013StabilityEquations,Hur2015ModulationalWaves,HJ_SurfaceTension}, the variations in the basis functions $\phi_j(\cdot;a,\xi)$ does not play a role in the asymptotic
calculation below as they contribute only to higher-order terms than what are needed here.  Thus, below, the calculations are done with the $\xi$-independent basis $\phi_j(\cdot;a,0)$.

Continuing our strategy, we now compute the action of $\mathcal T^\lambda_{a,\xi}$ on the critical subspace $\Sigma_{a,\xi}$.
Here $\langle\hspace{2px}\cdot\hspace{2px},\hspace{2px}\cdot\hspace{2px}\rangle$ denotes the $L^2(\mathbb{T})$- inner product as defined in \eqref{def:i-product}.
Now, for $\xi$ and $|a|$ sufficiently small, we expand $\mathcal{T}^\lambda_{a,\xi}$ using Baker-Campbell-Hausdorff formula as 
\begin{equation}\label{e:op}
    \mathcal T^\lambda_{a,\xi} = T_{0,a} + i\xi T_{1,a}-\dfrac{\xi^2}{2} T_{2,a} + O(\xi^3)
\end{equation}as $\xi\to 0$
where
\begin{align*}
 &T_{0,a} := \mathcal T^\lambda_{a,0}=T_0+2k^2\partial z^2(a\cos z+a^2A_2\cos 2z)+\mathcal{O}(a^3)\\
 &T_{1,a} := [\mathcal T^\lambda_{a,0},z]=T_1+4k^2\partial z(a\cos z+a^2A_2\cos 2z)+\mathcal{O}(a^3)\\
 &T_{2,a} := [[\mathcal T^\lambda_{a,0},z],z]=T_2+4k^2(a\cos z+a^2A_2\cos 2z)+\mathcal{O}(a^3)
\end{align*}
and
\[
T_0=\mathcal T^\lambda_{0,0},\quad T_1=[T_0,z]\quad \text{and}\quad T_2=[[T_0,z],z].
\vspace{3px}
\]
Note that $T_{1,a}$ and $T_{2,a}$ are well defined in $L^2(\mathbb{T})$. Now, to find the action of $T_0$, $T_1$ and $T_2$ on the generalized kernel, we use the expansion $\mathcal T^\lambda_{0,\xi}e^{inz}$ rather than computing tedious fourier series expansions. Moreover,
\[
\mathcal T^\lambda_{0,\xi}(\cos nz)=\mathcal T^\lambda_{0,\xi}\left(\dfrac{e^{inz}+e^{-inz}}{2}\right)
\quad\quad \text{and} \quad\quad
\mathcal T^\lambda_{0,\xi}(\sin nz)=\mathcal T^\lambda_{0,\xi}\left(\dfrac{e^{inz}-e^{-inz}}{2i}\right).
\]
Consequently,
\begin{align*}
T_0(\cos nz)=& -n\lambda \sin nz+\g (n^2-1)\cos nz+\b k^2n^2(m(k)-m(kn))\cos nz, \\
T_1(\cos nz)=& \lambda \cos nz+2n\g \sin nz+\b k^2(-n^2km^\prime(kn)+2n(m(k)\\& -m(kn))\sin nz,\\
T_2(\cos nz)=& -2\g \cos nz-2\b k^2(m(k)-m(kn)-2nkm^\prime(kn)\\& -\dfrac{n^2k^2}{2}m^{\prime\prime}(kn))\cos nz,
\end{align*}
and \begin{align*}
T_0(\sin nz)=& n\lambda \cos nz+\g (n^2-1)\sin nz+\b k^2n^2(m(k)-m(kn))\sin nz, \\
T_1(\sin nz)=& \lambda \sin nz-2n\g \cos nz-\b k^2(-n^2km^\prime(kn)+2n(m(k)\\& -m(kn))\cos nz,\\
T_2(\sin nz)=& -2\g \sin nz-2\b k^2(m(k)-m(kn)-2nkm^\prime(kn)\\& -\dfrac{n^2k^2}{2}m^{\prime\prime}(kn))\sin nz.
\end{align*}
Using these, we get
\begingroup
\allowdisplaybreaks
\begin{align*}\label{eq:in1}
\left<\mathcal T^\lambda_{a,\xi}\phi_1,\phi_1\right>=& \dfrac{1}{2}(i\xi\lambda(1+4a^2A_2^2)-\dfrac{\xi^2}{2}(-2\g(1+a^2A_2^2)+\b k^4(m^{\prime\prime}(k)\\& +16a^2A_2^2m^{\prime\prime}(2k))+4\b k^3(m^\prime(k)+8a^2A_2^2m^{\prime}(2k))))+\mathcal{O}(\xi^3+a^3),\\
\left<\mathcal T^\lambda_{a,\xi}\phi_1,\phi_2\right>=& \dfrac{1}{2}(\lambda(1+8a^2A_2^2)+i\xi(-2\g(1+2a^2A_2^2)\\
&+\b k^3(m^\prime(k)+16a^2A_2^2m^\prime(2k)))+\mathcal{O}(\xi^3+a^3),\\
\left<\mathcal T^\lambda_{a,\xi}\phi_2,\phi_1\right>=& \dfrac{1}{2}(-\lambda(1+8a^2A_2^2)+i\xi(2\g(1+2a^2A_2^2)-4a^2k^2A_2^2-\b k^3(m^\prime(k)+16a^2A_2^2m^\prime(2k)))\\
&+\mathcal{O}(\xi^3+a^3).,\\
\left<\mathcal T^\lambda_{a,\xi}\phi_2,\phi_2\right>=& \dfrac{1}{2}(-2a^2k^2A_2+i\xi\lambda(1+4a^2A_2^2)-\dfrac{\xi^2}{2}(-2\g(1+a^2A_2^2)+4a^2k^2A_2+\b k^4(m^{\prime\prime}(k)\\& +16a^2A_2^2m^{\prime\prime}(2k))+4\b k^3(m^\prime(k)+8a^2A_2^2m^{\prime}(2k))))+\mathcal{O}(\xi^3+a^3).
\end{align*}\endgroup
Using the above obtained expressions, along with the expansion \eqref{e:op}, it follows that the matrix $\mathcal{B}_{a,\xi}^\lambda$ in \eqref{eq:bmat1}
can be expanded for sufficiently small $\xi$ and $|a|$ as
\begingroup
\allowdisplaybreaks
\begin{align*}
\mathcal{B}^\lambda_{a,\xi} = 
&\dfrac{\lambda}{2}\begin{pmatrix}0&
1
\\
-1 &0
\end{pmatrix}+\dfrac{a^2}{2}\begin{pmatrix}0&
8\lambda A_2^2
\\
-8\lambda A_2^2 &-2k^2A_2
\end{pmatrix}+\dfrac{i\xi}{2}\begin{pmatrix}\lambda&
-2\g+\b k^3m^\prime (k)
\\
2\g-\b k^3m^\prime (k) &\lambda
\end{pmatrix}\\&+\dfrac{i\xi a^2}{2}\begin{pmatrix}4\lambda A_2^2&
-4\gamma A_2^2+16\b k^3A_2^2m^\prime (2k)
\\
4\gamma A_2^2-4k^2A_2^2-16\b k^3A_2^2m^\prime (2k) &4\lambda A_2^2
\end{pmatrix}\\&-\dfrac{\xi^2 }{2}(-2\g+\b k^4m^{\prime\prime}(k)+4\b k^3m^\prime(k))\begin{pmatrix}1&
0
\\
0&1
\end{pmatrix}\\&-\dfrac{\xi^2a^2 }{2}\left((-2\g A_2^2+16\b k^4A_2^2m^{\prime\prime}(2k)+32\b k^3A_2^2m^\prime(2k))\begin{pmatrix}1&
0
\\
0&1
\end{pmatrix}+\begin{pmatrix}
    0&0\\0&4k^2A_2
\end{pmatrix}\right)\\
&
+\mathcal{O}(\xi^3+a^3).
\end{align*}
\endgroup

To study the two critical values $\lambda(a,\xi)$ bifurcating from the $(\lambda,a,\xi)=(0,0,0)$ state, we recall from \eqref{eq:cheq1p} that we must
study the roots of the polynomial
\[
\det(\mathcal{B}^\lambda_{a,\xi})=b_0(a,\xi)+ib_1(a,\xi)\lambda+b_2(a,\xi)\lambda^2,
\]
where the coefficient functions $b_j$, defined for $|(a,\xi)|\ll 1$, depend smoothly on $a$ and $\xi$.  From the symmetry property in
Proposition \ref{p:1}(1), we further see that the functions $b_j$ are real-valued for $j=0,1,2$.  Similarly,
Proposition \ref{p:1}(2) implies that $b_0$ and $b_2$ are even functions of $\xi$ while $b_1$ is odd in $\xi$.  Note also that
since the values $\Omega_{1,\xi}$ and $\Omega_{-1,\xi}$ collide at $\lambda=0$ when $\xi=0$, we see that $b_0=\mathcal{O}(\xi^2)$ for 
 $|(a,\xi)|\ll 1$.  It follows that
\[
b_j(a,\xi)=d_j(a,\xi)\xi^{2-j},~~j=0,1,2
\]
where the functions $d_j$ are real-valued functions depending smoothly on $a$ and $\xi$ for $|(a,\xi)|\ll 1$.  Setting $\lambda=i\xi X$ it follows that
\[
\det(\mathcal{B}^\lambda_{a,\xi})=\xi^2\left(d_0(a,\xi)-d_1(a,\xi)X-d_2(a,\xi)X^2\right)=:\xi^2 Q(a,\xi,X).
\]
The underlying wave is thus modulationally unstable if the polynomial $Q$ admits roots with non-zero imaginary parts, while if it modulationally stable
if $Q$ admits two distinct real roots.

To determine the reality of the roots of $Q$, it is sufficient to study its discriminant $\mathcal{D}_{a,\xi}$, which can be directly expanded as
\begin{align*}
\mathcal{D}_{a,\xi}&=    \xi^2(\b k^3(km^{\prime\prime}(k)+2m^\prime(k))+2\g)^2+a^2\dfrac{k^4}{2}\left(\dfrac{2\g+\b k^3(km^{\prime\prime}(k)+2m^\prime(k))}{3\g+4\b k^2(m(k)-m(2k))}\right)\\
   & \quad+\mathcal{O}\left(a^2\left(a^2+\xi^2\right)\right).
\end{align*}
It follows that the asymptotically small background periodic traveling waves $w(\cdot;a,k)$ are modulationally stable provided that $\mathcal{D}_{a,\xi}>0$ for
$0<|\xi|\ll 1$ and modulationally unstable if $\mathcal{D}_{a,\xi}<0$ for $0<|\xi|\ll 1$.
In particular, we note that for a fixed small $|a|$ we can choose $0<\xi_0\ll 1$ sufficiently small such that $\mathcal{D}_{a,\xi}<0$ for all $0<|\xi|\ll\xi_0$, indicating
modulational instability of the background wave, provided that
\begin{equation}\label{e:condition}
\dfrac{2\g+\b k^3(km^{\prime\prime}(k)+2m^\prime(k))}{3\g+4\b k^2(m(k)-m(2k))}<0.
\end{equation}
while one can similarly guarantee $\mathcal{D}_{a,\xi}>0$ for all $0<|\xi|\ll 1$ provided that the expression in \eqref{e:condition} is strictly positive.

To complete the proof of Theorem \ref{t:1}, it remains simply to note by the calculations in Remark \ref{R:Lighthill2}, along with
the observation that
\[
\frac{dc_g}{dk}(k)=\frac{2\gamma}{k^3}+\beta\left(2m'(k)+km''(k)\right),
\] 
that the quantity in \eqref{e:condition} can be rewritten as
\[
\dfrac{2\g+\b k^3(km^{\prime\prime}(k)+2m^\prime(k))}{3\g+4\b k^2(m(k)-m(2k))}=\frac{4}{k^3\left(c_p(k)-c_p(2k)\right)}\frac{dc_g}{dk}(k).
\]
It follows that the sign of the expression in \eqref{e:condition} agrees precisely with that of the modulational instability index $\Delta(k)$ in Theorem \ref{t:1}, completing
the proof.

\section{Application to Specific Models}\label{s:4}

In this section, we apply the general result from Theorem \ref{t:1}  to a number of specific models.  When possible, we compare our results to previously
known results.  Specifically, we first apply our results to the classical Ostrovsky equation \eqref{E:Ost} as well as the fractional Ostrovsky equation.
We then consider the Whitham-Ostrovsky equation, where the associated one-dimensinoal equation encodes the full dispersion relation from the Euler equations
for uni-directional surface water waves, as well as an Ostrovsky variant of the well-studied Intermediate Long-Wave equation.  Our final examples go
further to consider the effects of adding capillary effects into the the classical Ostrovsky and Whitham-Ostrovsky equations.

We note that for each of the examples considered in this section, the corresponding modulational stability of
small amplitude periodic traveling waves in the non-rotational version of the equation given by \eqref{e:whitham2} has been previously studied: see, for example,
the works \cite{HK08,HJ_SurfaceTension,Hur2015ModulationalWaves,Johnson2013StabilityEquations}.  As such, the results
presented below essentially study the effect of rotation (as modeled by the Ostrovsky equation) on these previously studied models.

\subsection{Classical Ostrovsky Equation}\label{ss:1}

As a first application, we apply Theorem \ref{t:1} to the classical Ostrovsky equation \eqref{E:Ost}.  Note that \eqref{E:Ost} corresponds to our generalized-dispersion
model \eqref{eq:gost} with the choice
\begin{equation}\label{e:mos}
\mathcal{M} = 1+\partial_x^2,~~{\rm i.e.}~~m(k)=1-k^2.
\end{equation}
The symbol $m(k)$ clearly satisfies Hypotheses~\ref{h:m}~(H1), (H2) ($\alpha=2$, $C_1=1$ and $C_2=2$), and (H3) ($m$ is strictly decreasing for $k>0$). 
Consequently, we obtain asymptotically small periodic traveling wave solutions from Theorem \ref{T:sol} along with asymptotic expansions provided
explicitly by substituting $m(k)=1-k^2$ into \eqref{E:w_ansatz}.

In this case, the corresponding phase and group velocities are given explicitly by
\[
c_p(k) = \beta\left(1-k^2\right)+\frac{\gamma}{k^2}~~{\rm and}~~c_g(k) = \beta\left(1-3k^2\right)-\frac{\gamma}{k^2}.
\]
The qualitative properties clearly depend on the sign of $\beta$.  For $\beta>0$, the group velocity $c_g$  attains a global maxima (for $k>0$) at  $k=k_c=\left(\dfrac{\g}{3|\b|}\right)^{1/4}$
and is monotonically increasing for $k\in(0,k_c)$ and monotonically decreasing for $k>k_c$.  Further, in this case the phase speed $c_p(k)$ is strictly decreasing for $k>0$,
and hence in this case one has
\[
\Delta(k)>0~~{\rm for}~~k\in(0,k_c)~~~{\rm and}~~~\Delta(k)<0~~{\rm for}~~k>k_c.
\]
This establishes the modulational instability and stability result in Corollary \ref{c:1} in the case $\beta>0$.  As noted in the Introduction, 
in the case $\beta=1$ and $\gamma=1$ our result agrees  with that derived by Whitfield and Johnson in \cite{Whitfield2014Rotation-inducedWaves}.

Similarly, when $\beta<0$ the group velocity is strictly increasing for all $k>0$, while $c_p(k)-c_p(2k)$ changes signs exactly once, from positive to negative, at $k_c=\left(\dfrac{\g}{4|\b|}\right)^{1/4}$.
This establishes Corollary \ref{c:1} in the case $\beta<0$. Notice, in particular, that the mechanisms accounting for the modulational instabilities in the case $\beta>0$ are different
from those in the $\beta<0$ case.

\subsection{Fractional Ostrovsky Equation}

The Ostrovsky-fractional KdV equation 
\begin{equation}\label{eq:OfKdV}
\left( u_t+\beta(1-|\partial_x|^\delta) u_x+(u^2)_x\right)_x-\gamma u=0,~ \gamma>0, ~\beta \in \R\setminus\{0\}
\end{equation}
can be obtained from \eqref{eq:gost} by choosing $m(k)=1-|k|^\delta$, $\delta>1/2$.  In this case, the symbol $m(k)$ clearly satisfies Hypotheses~\ref{h:m}~(H1), (H2) ($\delta=\alpha$, $C_1=1$, and $C_2=2$), and (H3) ($m$ is strictly decreasing for $k>0$). As above, we can obtain  asymptotically small amplitude  periodic travleing wave solutions of the Ostrovsky-fKdV equation from 
Theorem \ref{T:sol} by substituting $m(k)=1-|k|^\delta$.  Applying precisely the same reasoning as in the previous section for the classical Ostrovsky equation, we obtain the following
result.

 \begin{corollary}
    For a fixed $\g>0$, a $2\pi/k$-periodic traveling wave of the Ostrovsky-fKdV equation \eqref{eq:OfKdV} with sufficiently small amplitude is modulationally unstable if $k>k_c$, where
   \begin{equation*}
\begin{cases}
k_c = \left(\dfrac{2\g}{\delta(1+\delta)|\b|}\right)^{1/(2+\delta)} & \quad \text{ if } \b >0, \\
k_c = \left(\dfrac{3\g}{4(2^\delta-1)|\b|}\right)^{1/(2+\delta)} & \quad \text{ if } \b <0. 
\end{cases}
\end{equation*}
and it is modulationally stable otherwise.
 \end{corollary}
 \begin{proof}
    The proof is same as the Ostrovsky equation in Section~\ref{ss:1}
\end{proof}

\subsection{Whitham-Ostrovsky Equation}

Continuing as above, we may consider an Ostrovsky variant of the well-studied Whitham equation \eqref{e:whitham} by choosing
 \begin{equation}\label{e:mwh}
 m(k)=\sqrt{\frac{\tanh k}{k}}
 \end{equation} 
The symbol $m(k)$ clearly satisfies Hypotheses~\ref{h:m}~(H1), (H2) ($\alpha=-1/2$, $C_1=1$ and $C_2=2$), and (H3) ($m$ is strictly decreasing for $k>0$). 
As above, we can obtain  asymptotically small amplitude  periodic travleing wave solutions of the Ostrovsky-fKdV equation from 
Theorem \ref{T:sol}.  In this case, our Theorem \ref{t:1} gives the following result.

 \begin{corollary}\label{C:WO}
    For a fixed $\g>0$, a $2\pi/k$-periodic traveling wave of the Whitham-Ostrovsky equation \eqref{E:Ost} with sufficiently 
    small amplitude is modulationally unstable if $k>k_c$, where $k_c$ is the unique real solution of following equations
   \begin{equation*}
\begin{cases}
k^3\left(k\dfrac{d^2}{dk^2}\left(\sqrt{\frac{\tanh k}{k}}\right)+2\dfrac{d}{dk}\left(\sqrt{\frac{\tanh k}{k}}\right)\right)=-\dfrac{2\g}{|\b|}
 & \quad \text{ if } \b > 0, \\~\\
k^2\left(\sqrt{\frac{\tanh k}{k}}-\sqrt{\frac{\tanh 2k}{2k}}\right)=\dfrac{3\g}{4|\b|} & \quad \text{ if } \b < 0 
\end{cases}
\end{equation*}
and it is modulationally stable otherwise.
 \end{corollary}
 \begin{proof}
    
  The proof is same as the Ostrovsky equation in Section~\ref{ss:1}
    \end{proof}

We note that numerics indicate that the functions
\begin{equation}\label{e:Whitham_Ost1}
k\mapsto -k^3\left(k\dfrac{d^2}{dk^2}\left(\sqrt{\frac{\tanh k}{k}}\right)+2\dfrac{d}{dk}\left(\sqrt{\frac{\tanh k}{k}}\right)\right)
\end{equation}
and
\begin{equation}\label{e:Whitham_Ost2}
k\mapsto k^2\left(\sqrt{\frac{\tanh k}{k}}-\sqrt{\frac{\tanh 2k}{2k}}\right)
\end{equation}
are both equal to zero at $k=0$ and are both monotonically increasing for $k>0$, tending to infinity as $k\to\infty$: see Figure \ref{F:Whitham_Ost}.  
Consequently, it is clear in each case $\beta>0$ and $\beta<0$ there is a unique $k=k_c(\beta)>0$ where the conditions in Corollary \ref{C:WO} are satisfied.
Further, they are both $\mathcal{O}(k^4)$ for $|k|\ll 1$
and hence, for the Whitham-Ostrovsky equation, the critical frequency satisfies $k_c\sim\gamma^{1/4}$ for $0<\gamma\ll 1$.  Note that for the one-dimensional 
Whitham equation for water waves, corresponding to \eqref{e:whitham2} with the choice \eqref{e:mwh}, the corresponding asymptotically small waves
exhibit a modulational instability for $k>\widetilde{k}_c\approx 1.146$.  Again, this demonstrates a singularity of the Whitham-Ostrovsky equation
in the limit $\gamma\to 0^+$: see Remark \ref{r:singular_limit1} in the Introduction.

\begin{figure}[ht]
(a) \includegraphics[scale=0.5]{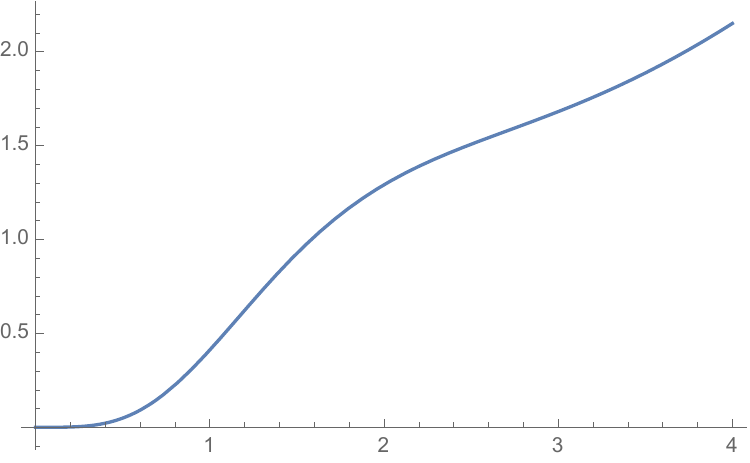}\qquad (b) \includegraphics[scale=0.5]{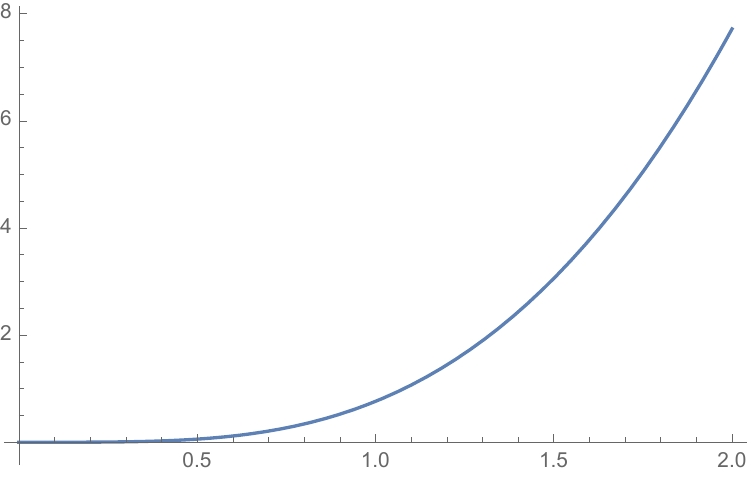}
\caption{Plots of the functions (a) \eqref{e:Whitham_Ost1} and (b) \eqref{e:Whitham_Ost2} associated with the Whitham-Ostrovsky equation.
}\label{F:Whitham_Ost}
\end{figure}

\subsection{ILW-Ostrovsky Equation}

The Intermediate  Long Wave (ILW) equation is given by
\[
u_t+\beta \mathcal{M}u_x+\left(u^2\right)_x=0
\]
where here $\mathcal{M}$ is a Fourier multiplier with symbol $m(k)=k\coth(k)$.	The ILW is well-known to describe long internal gravity waves in a two-layer stratified fluid, with the lower layer
having a large finite depth.  By adding rotational effects, we can obtain an ILW-Ostrovsky equation 
by making the choice $m(k)=k\coth(k)$ in \eqref{eq:gost}.  
The symbol $m(k)$ clearly satisfies Hypotheses~\ref{h:m}~(H1), (H2) ($\alpha=2$, $C_1=1$ and $C_2=2$), and (H3) ($m$ is strictly increasing for $k>0$). 
As above, we can obtain  asymptotically small amplitude  periodic travleing wave solutions of the ILW-Ostrovsky equation from 
Theorem \ref{T:sol}.
Again applying the same reasoning as in Section \ref{ss:1}, we obtain the following result.

 \begin{corollary}\label{C:ILW_OST}
    For a fixed $\g>0$, a $2\pi/k$-periodic traveling wave of the ILW-Ostrovsky equation \eqref{E:Ost} with sufficiently small amplitude is modulationally unstable if $k>k_c$, where $k_c$ is the unique real solution of following equations
   \begin{equation*}
\begin{cases}
k^2(2k \coth{2k}-k \coth{k})-\dfrac{3\g}{4|\b|}=0 & \quad \text{ if } \b > 0, \\~\\
k^3(-4k\csch^2{k}+2k^2\coth{k}\csch^2{k}+2\coth{k})-\dfrac{2\g}{|\b|}=0 & \quad \text{ if } \b < 0 
\end{cases}
\end{equation*}
and it is modulationally stable otherwise.
 \end{corollary}
 \begin{proof}
   The proof is same as the classical Ostrovsky equation in Section~\ref{ss:1}.  Note that while we lack an explicit formula for the critical frequency $k_c$,
   it can of course be numerically approximated.
    \end{proof}

We note that, similar to the above example, numerics indicate that  the functions
\begin{equation}\label{e:ILW_OST1}
k\mapsto k^2(2k \coth{2k}-k \coth{k})
\end{equation}
and
\begin{equation}\label{e:ILW_OST2}
k\mapsto k^3(-4k\csch^2{k}+2k^2\coth{k}\csch^2{k}+2\coth{k})
\end{equation}
both vanish at $k=0$ and are monotonically increasing for $k>0$, tending to infinity as $k\to\infty$: see Figure \ref{F:ILW_OST}.  As above, it follows
that for each case $\beta>0$ and $\beta<0$ there is a unique $k=k_c(\beta)>0$ where the conditions in Corollary \ref{C:ILW_OST}  are satisfied. 
Further, both functions are $\mathcal{O}(k^4)$ for $|k|\ll 1$.  
As such, the for the ILW-Ostrovsky equation we again see the scaling relation $k_c\sim\gamma^{1/4}$ for $0<\gamma\ll 1$
as observed for both the classical Ostrovsky and Whitham-Ostrovsky equations considered above.

\begin{figure}[ht]
(a) \includegraphics[scale=0.5]{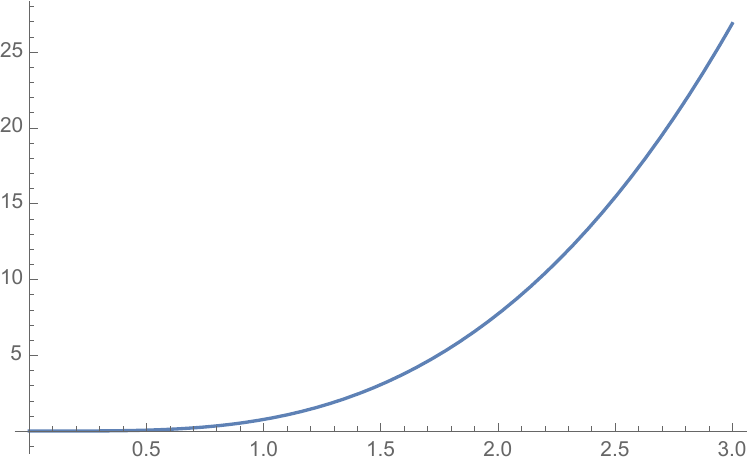}\qquad (b) \includegraphics[scale=0.5]{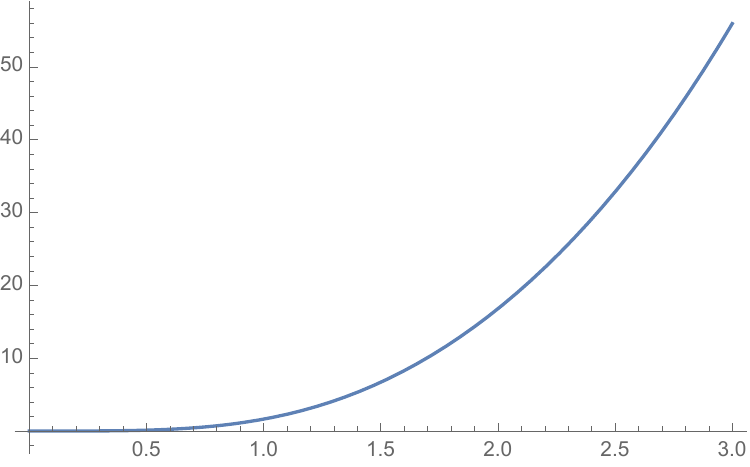}
\caption{Plots of the functions (a) \eqref{e:ILW_OST1} and (b) \eqref{e:ILW_OST2} associated with the ILW-Ostrovsky equation.
}\label{F:ILW_OST}
\end{figure}

\subsection{Effects of Surface Tension on Modulational instability}

We conclude our study by further adding capillary effects to the classical Ostrovsky and Whtiham-Ostrovsky equations considered above.  Note that it is known
that incorporating such capillary effects into even the non-rotational models drastically affects the modulational instability of asymptotically small waves:
see, for example, \cite{HJ_SurfaceTension}.  As such, we find it worthwhile to report on the effects adding surface tension to the rotational models considered
in our work.

\subsubsection{Classical Ostrovsky Equation with Surface Tension} 

Upon normalization, in the presence of surface tension in the classical Ostrovsky equation, the dispersive symbol in \eqref{e:M} will change to
\begin{equation}\label{e:mos1}
m(k,T)=1-k^2+3Tk^2
\end{equation} 
where here $T\geq0$ is the constant coefficient of surface tension. Note that  for $T=0$, \eqref{e:mos1} reduces to \eqref{e:mos}, which was studied 
in Section \ref{ss:1} above.
 
Using Theorem~\ref{t:1}, we have following results on the effects of surface tension on modulational stability in the Ostrovsky equation.  As one might expect,
the results depend on the sign of $\beta$ as well as whether the parameter $T$ satisfies $0<T<1/3$ or $T>1/3$.

\begin{corollary}\label{cor:OT}
In the case $\b>0$, a sufficiently small $2\pi/k$-periodic traveling wave of the full-dispersion Ostrovsky equation  \eqref{e:whitham2} with dispersive symbol
$m(k)=m(k;T)$ as in \eqref{e:mos1} is modulationally unstable 
if $k>k_c(T)$, where
\begin{equation*}
\begin{cases}
k_c(T)=\left(\dfrac{\g}{3|\b| (1-3T)}\right)^{1/4} & \quad \text{ if } T<1/3, \\
k_c(T)=\left(\dfrac{\g}{4|\b| (3T-1)}\right)^{1/4} & \quad \text{ if } T>1/3,
\end{cases}
\end{equation*}
and it is modulationally stable otherwise.\\

For the same equation, but in the case $\b<0$, a sufficiently small $2\pi/k$-periodic traveling wave is modulationally unstable if $k>k_c(T)$, 
where
\begin{equation*}
\begin{cases}
k_c(T)=\left(\dfrac{\g}{4|\b| (1-3T)}\right)^{1/4} & \quad \text{ if } T<1/3, \\
k_c(T)=\left(\dfrac{\g}{3|\b| (3T-1)}\right)^{1/4} & \quad \text{ if } T>1/3,
\end{cases}
\end{equation*}
and it is modulationally stable otherwise.
\end{corollary}

For $T=0$, that is, when capillary effects are absent, value of $k_c(0)$ of course agrees with $k_c$ in Corollary~\ref{c:1}. For $T > 0$, we describe the modulational instability through
the Figure~\ref{fig:lem4.3}, where we took $\frac{\gamma}{\beta}=1$ and $\frac{\gamma}{\beta}=-1$ for the computations\footnote{Note that 
while the results for the cases 
$\frac{\gamma}{\beta}=\pm 1$ are qualitatively similar, the exact stability boundary curves in each figure are slightly different.}.
 In $k$-$k\sqrt{T}$ plane,  two curves are corresponding to each mechanism splitting the plane into two regions of stability and two regions of instability. Any fixed $T>0$
corresponds to a line passing through the origin of slope $\sqrt{T}$. For $T\neq1/3$, the line through the origin only crosses one curve at a time producing one interval of stable wave numbers and one  interval of unstable wave numbers. The result becomes inconclusive for $T=1/3$.
\begin{figure}[ht]
    \centering
    {(a)\includegraphics[width=8cm]{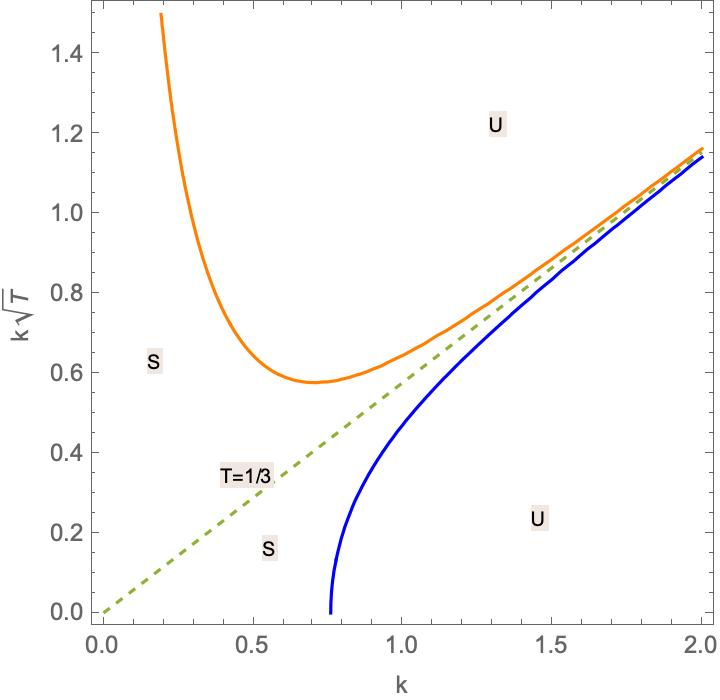}~~(b)\includegraphics[width=8cm]{3_2.jpeg} }%
    
    \caption{Stability diagrams for sufficiently small, periodic wave trains of Ostrovsky equation with surface tension (see Corollary \ref{cor:OT}).  
    Here, for the computations we took (a) $\frac{\gamma}{\beta}=1$ and (b) $\frac{\gamma}{\beta}=-1$.  In both figures,  ``S" and ``U" denote stable and unstable regions. Solid curves represent roots of the stability condition \eqref{e:cd}.  
Specifically,  the  solid orange curve represents the coots of $c_p(k)-c_p(2k)$ while the solid blue curve represents the roots of $ \frac{dc_g(k)}{dk}$.  (Color available online.)}
    \label{fig:lem4.3}
\end{figure}

\subsubsection{Whitham-Ostrovsky Equation with Surface Tension} To incorporate the effects of surface tension into the Whitham-Ostrovsky equation we replace  the 
dispersive symbol $m(k)$ in \eqref{e:mwh} by
 \begin{equation}\label{e:mwh1}
 m(k,T)=\sqrt{\dfrac{\tanh{k}}{k}(1+Tk^2)},
 \end{equation}
where $T\geq0$ is the (properly normalized) coefficient of surface tension. Note that when $T=0$, \eqref{e:mwh1} of course reduces to \eqref{e:mwh}, and hence we will
only consider the case $T>0$ here.  

First, note that for each fixed $T>0$ the symbol $m(\cdot,T)$ satisfies 

\[
m(0,T)=1\quad{\rm and}\quad \lim_{|k|\to\infty}m(k,T)=\infty.
\]
Further, it is readily seen that $m(\cdot,T)$ is strictly monotone for $k>0$ provided $T>1/3$, while for $0<T<1/3$ there exists a unique $k^*(T)$ such that $m(\cdot,T)$ 
is monotone decreasing on $(0,k^*(T))$ and monotone increasing for $(k^*(T),\infty)$.  In this way, the critical surface tension $T=1/3$ is often used to differentiate between
the ``weak" ($0<T<1/3$) and  ``strong" ($T>1/3$) surface tension cases.

According to Theorem~\ref{t:1}, whenever factors $c_p(k)-c_p(2k)$ and $\frac{dc_g(k)}{dk}$ of the modulational instability index $\Delta(k)$ are zero, stability changes, where $c_k(p)$ and $c_g(k)$ are defined in \eqref{e:gvel}. To show the explicit dependence on surface tension, we replace $c_p(k)$ and $c_g(k)$, by $c_p(k,T)$ and $c_g(k, T)$, respectively. 
Notice in particular that the roots (in the frequency $k$) of the factors $c_p(k,T)-c_p(2k,T)$ and $\frac{dc_g(k,T)}{dk}$ are precisely the same for each fixed value of the ratio 
\[
\sigma:=\frac{\g}{\beta}.
\]
Of course, since $\g$ is always positive the sign of $\sigma$ agrees with that of $\beta$. Below, we will consider the cases $\sigma>0$ and $\sigma<0$  separately.  

\ 

\noindent
\underline{\bf Case $\sigma>0$:} \\

It is  straightforward to see that for each $\sigma>0$, there exists a unique $T=T_c(\sigma)$ such that the group velocity $c_g$ attains a local maxima and a local minima for $0<T<T_c(\sigma)$ and is monotonic for $T>T_c(\sigma)$. For instance, for $\sigma=0.1$, $T_c \approx 0.132$ and therefore, plots of $c_g(k,T)$ vs. $k$ for $T=0.02 < T_c$ and $T=0.5>T_c$ (see Figure~\ref{fig:np1}) confirm the monotonocity property of $c_g$. Moreover, for any $T>0$, $c_p(k,T)-c_p(2k,T)$ is changing its sign only once, see Figure~\ref{f:cp}, for instance. 
\begin{figure}[ht]
    \centering
    
    \subfloat[\centering $T=0.02$]{{\includegraphics[width=7cm]{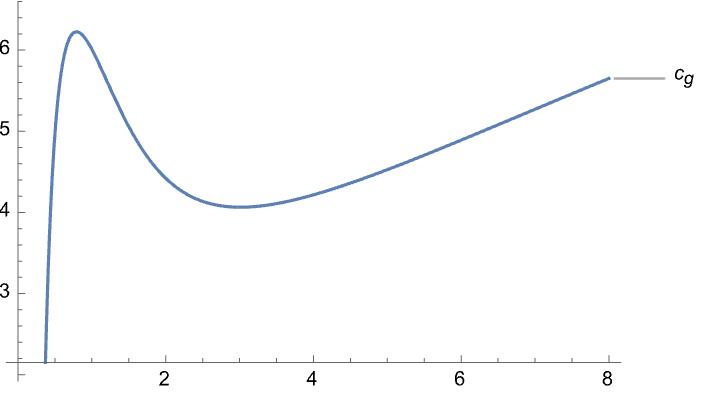} }}
    \qquad
    \subfloat[\centering $T=0.5$]{{\includegraphics[width=7cm]{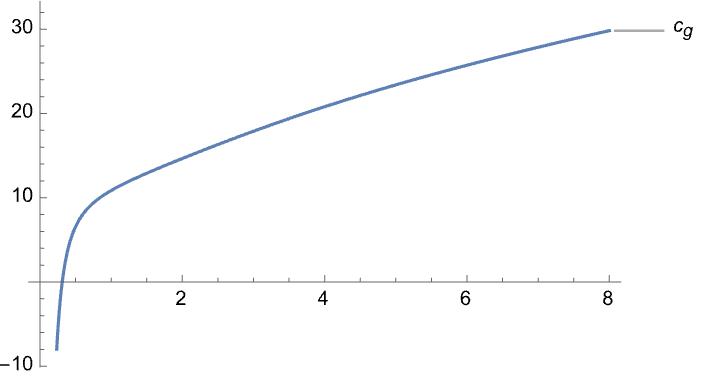} }}
    \caption{Graph of function $c_g$ vs. $k$ for $\sigma=0.1$ for which $T_c\approx 0.132$.}\label{fig:np1}%
    \end{figure}
\begin{figure*}[ht]
    \centering
    \includegraphics[width=0.49\textwidth]{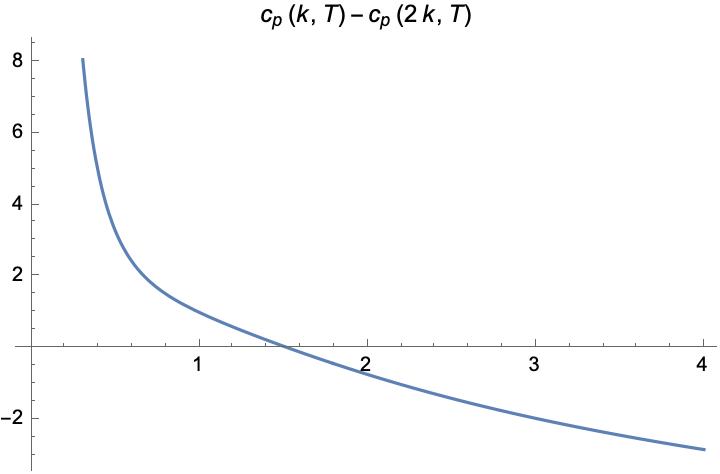}
    \caption{$c_p(k,T)-c_p(2k,T)$ vs $k$ for $\sigma=0.1$ and 
    $T=0.2$.}\label{f:cp}
\end{figure*}
Therefore, for $0<T<T_c(\sigma)$, there exist three critical wavenumbers $0<k_1<k_2<k_3$ such that sufficiently small $2\pi/k$-periodic traveling wave is modulationally
stable for $k\in(0,k_1)\cup (k_2,k_3)$ and modulationally
unstable for $k\in(k_1,k_2)\cup (k_3,\infty)$. 
On the other hand, for $T>T_c(\sigma)$ there is only one critical wave number $k_c>0$ such that sufficiently small
$2\pi/k$-periodic traveling wave is modulationally
stable for $0<k<k_c(T)$ and modulationally unstable for $k > k_c (T )$.
The graph of $T_c$  vs. $\sigma$ is shown in Figure \ref{fig:lem4.9}.

In \cite{Kaw} and \cite{DR}, the effects of surface tension on modulational instability in the full water wave problem has been shown in $k-k\sqrt{T}$ plane. We produce a similar plot for $\sigma  = 0.1$ in Figure~\ref{fig:2}. In $k$-$k\sqrt{T}$ plane, two curves are corresponding to each mechanism splitting the plane into three regions of stability and three regions of instability. For $\sigma=0.1$, this can be easily seen that $T_c\approx0.132$ such that for $0<T<0.132$, the line crosses both the curves producing two intervals of stable wave numbers and two  intervals of unstable wave numbers. On the other hand, for $T>0.132$, the line through the origin crosses only one curve producing one interval of stable wave numbers and one interval of unstable wave numbers. 
\begin{figure}[ht]
    \centering
    {\includegraphics[width=10cm]{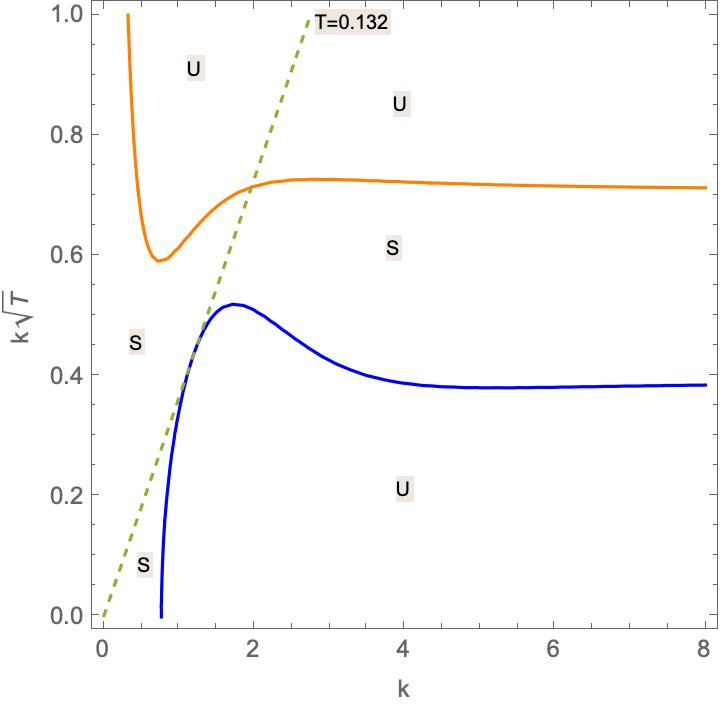} }%
    
    \caption{Stability diagram for sufficiently small, periodic wave trains Whitham-Ostrovsky equation for $\sigma=0.1$. ``S" and ``U" denote stable and unstable regions. 
    As in Figure \ref{fig:lem4.3}, the solid orange curve represent roots of $c_p(k)-c_p(2k)$ and the solid blue curve represent roots of $\dfrac{dc_g}{dk}$.  
    (Color available online.)  Note in this case we have $T_c\approx0.132$.  }
    \label{fig:2}
\end{figure}

\

\noindent
\underline{\bf Case $\sigma<0$:}\\

For each $\sigma<0$, there exists $T=T_c(\sigma)$ such that for $0<T<T_c(\sigma)$, $c_p(k,T)-c_p(2k,T)$ 
changes its sign at two wavenumbers resulting in two critical wavenumbers, and for $T>T_c(\sigma)$, $c_p(k,T)-c_p(2k,T)$ does not change its sign (see Figure \ref{fig:np2}). On the other hand, for each $T>0$, $c_g$ attains extremum only once (see Figure \ref{f:cg}) producing one critical wavenumber. Therefore, for $0<T<T_c(\sigma)$, there exist three critical wavenumbers such that sufficiently small $2\pi/k$-periodic traveling wave is modulationally
stable for $k\in(0,k_{c_1})\cup (k_{c_2},k_{c_3})$ and modulationally
unstable for $k\in(k_{c_1},k_{c_2})\cup (k_{c_3},\infty)$. For $T>T_c(\sigma)$, there is only one critical wave number $k_c$ such that sufficiently small
$2\pi/k$-periodic traveling wave is modulationally
stable for $0<k<k_c(T)$ and modulationally
unstable for $k > k_c (T )$. The graph of $T_c$ vs. $\sigma$ is shown in Figure~\ref{fig:lem4.9}. 
\begin{figure}[ht]
    \centering
    
    \subfloat[\centering $T=0.02$]{{\includegraphics[width=7cm]{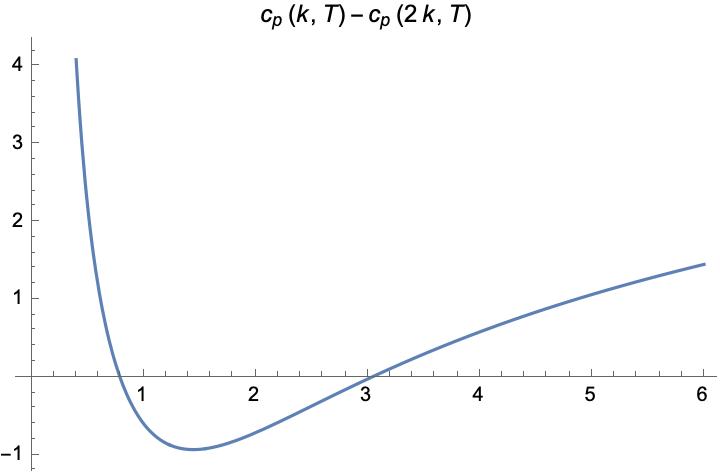} }}
    \qquad
    \subfloat[\centering $T=0.5$]{{\includegraphics[width=7cm]{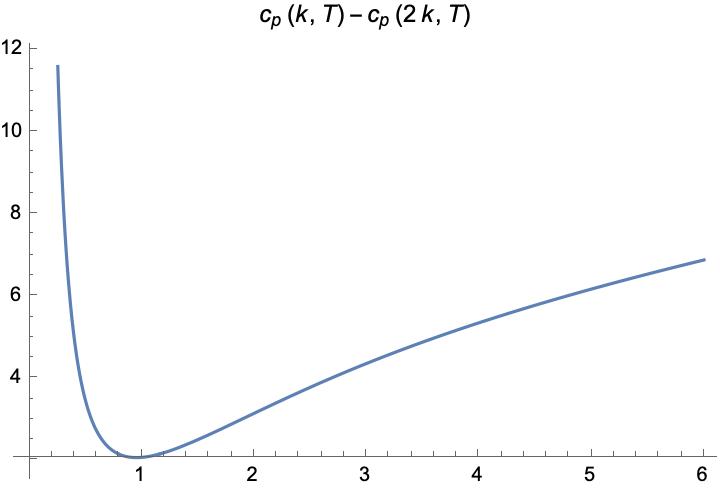} }}
    \caption{Graph of function $c_p(k,T)-c_p(2k,T)$ vs. $k$ for $\sigma=-0.1$ for which $T_c=0.141$.}\label{fig:np2}%
    \end{figure}
\begin{figure*}[ht]
    \centering
    \includegraphics[width=0.49\textwidth]{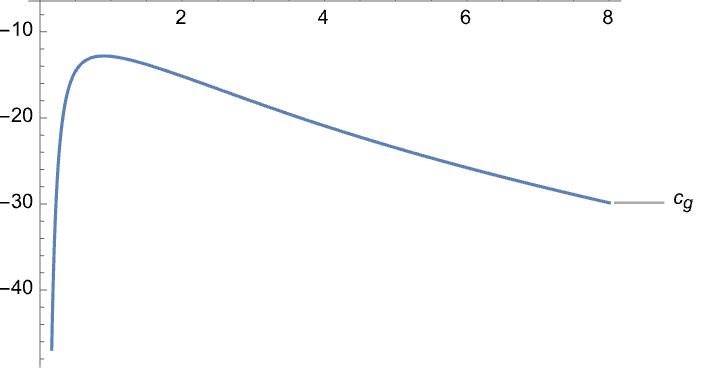}
    \caption{$c_g$ vs $k$ for $\sigma = -0.1$ and $T=0.5$.}\label{f:cg}
\end{figure*}

\begin{figure}[ht]
    \centering
    {\includegraphics[width=7cm]{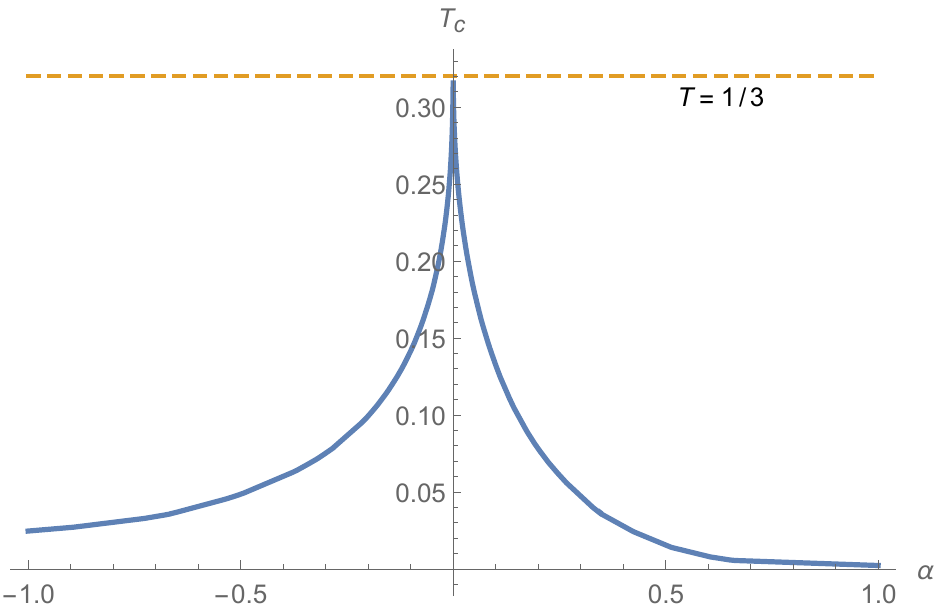} }%
    
    \caption{$T_c$ vs $\sigma$.}
    \label{fig:lem4.9}
\end{figure}

For $\sigma=-0.1$, we can see this behavior through the Figure~\ref{fig:3}. In $k$-$k\sqrt{T}$ plane, two curves are corresponding to each mechanism splitting the plane into three regions of stability and three regions of instability. For $\sigma=-0.1$, $T_c\approx0.141$ such that for $0<T<0.141$, the line crosses both the curves producing two intervals of stable wave numbers and two intervals of unstable wave numbers. On the other hand, for $T>0.141$, the line through the origin crosses only one curve producing one interval of stable wave numbers and one interval of unstable wave numbers. 
\begin{figure}[ht]
    \centering
    {\includegraphics[width=10cm]{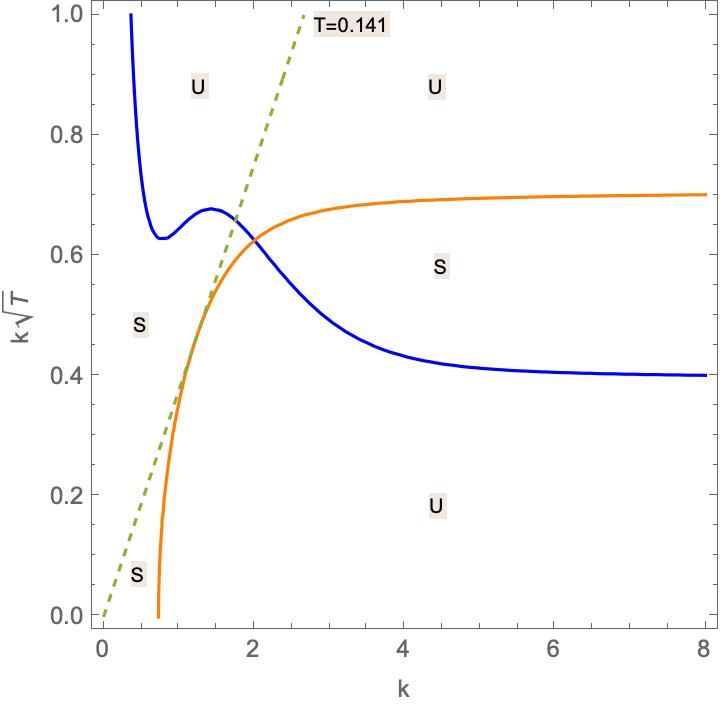} }%
    
    \caption{Stability diagram for sufficiently small, periodic wave trains of Whitham-Ostrovsky equation for $\sigma=-0.1$. ``S" and ``U" denote stable and unstable regions. Orange curve represent roots of $c_p(k,T)-c_p(2k,T)=0$ and blue curve represent roots of $\dfrac{dc_g}{dk}=0$.  }
    \label{fig:3}
\end{figure}
\begin{remark}
    For every $\sigma<0$, there is a value of $0<T=T_s<T_c(\sigma)$ corresponding to the intersection of two curves in the Figure~\ref{fig:3} for which there is only one interval of stable and unstable wavenumbers contrary to other values of $T$ in the interval $(0,T_c)$.   
\end{remark}

\bibliographystyle{abbrv}
\bibliography{ost.bib}

\end{document}